\documentclass[aos, reqno, preprint]{imsart}%
\RequirePackage{amsthm, amsmath, natbib, amsfonts, amssymb}%
\RequirePackage[OT1]{fontenc}%
\usepackage{graphicx, color}%
\usepackage{tikz}%
\usepackage{natbib}%

\numberwithin{equation}{section}%
\theoremstyle{plain}%

\definecolor{darkblue}{rgb}{0.0,0.0,0.7}

\RequirePackage[%
colorlinks = true,%
linkcolor = darkblue,%
citecolor = darkblue,%
urlcolor = darkblue, %
]{hyperref}%

\hypersetup{%
  pdfauthor = {St\'ephane Ga\"iffas, Guillaume Lecu\'e},%
  pdftitle = {Adaptive estimation of the regression with an assumption
    free design},%
  pdfcreator = {pdflatex},%
  pdfproducer = {pdflatex}}

\startlocaldefs

\newcommand \cB{{\cal B}}

\newcommand \cF{{\cal F}}

\newcommand \cH{{\cal H}}

\newcommand \cN{{\cal N}}

\newcommand \cX{{\cal X}}

\newcommand \R{{\mathbb  R}}

\newcommand{\var}{\text{Var}}%
\newcommand{\prodsca}[2]{\langle #1,#2 \rangle}%
\newcommand{\norm}[1]{\|#1\|}%

\newcommand{\bs}{\boldsymbol}

\DeclareMathOperator*{\argmin}{argmin}

\DeclareMathOperator{\pen}{pen}

\newcommand{\1}{{\rm 1}\kern-0.24em{\rm I}}

\renewcommand{\hat}{\widehat}

\newtheorem{theorem}{Theorem}%
\newtheorem{corollary}{Corollary}%
\newtheorem{lemma}{Lemma}%
\theoremstyle{remark}%
\newtheorem*{remark}{Remark}%
\newtheorem{definition}{Definition}%
\newtheorem*{assumption}{Assumption}%

\endlocaldefs

\begin{document}

\begin{frontmatter}

  \title{Aggregation of penalized empirical risk minimizers in
    regression}%
  \runtitle{Aggregation of penalized empirical risk minimizers}

  \begin{aug}
    \author{\fnms{St\'ephane} \snm{Ga\"iffas}
      \ead[label=e1]{stephane.gaiffas@upmc.fr}} and
    \author{\fnms{ Guillaume} \snm{Lecu\'e}
      \ead[label=e2]{lecue@latp.univ-mrs.fr}}

    \runauthor{S. Ga\"iffas and G. Lecu\'e} \affiliation{Universit\'e
      Paris~6 and CNRS, LATP Marseille}

    \address{Universit\'e Paris 6  \\
      Laboratoire de Statistique Th\'eorique et Appliqu\'ee \\
      175 rue du Chevaleret \\
      75013 Paris \\
      \printead{e1}}

    \address{ Laboratoire d'abalyse, topologie et probabilit\'e\\
	   Centre de Mathématiques et Informatique\\
	  Technopôle de Château-Gombert\\
	  39 rue F. Joliot Curie\\
	  13453 Marseille Cedex 13\\
	  France\\
      \printead{e2}}
  \end{aug}

  \begin{abstract}
    We give a general result concerning the rates of convergence of
    penalized empirical risk minimizers (PERM) in the regression
    model. Then, we consider the problem of agnostic learning of the
    regression, and give in this context an oracle inequality and a
    lower bound for PERM over a finite class. These results hold for a
    general multivariate random design, the only assumption being the
    compactness of the support of its law (allowing discrete
    distributions for instance). Then, using these results, we
    construct adaptive estimators. We consider as examples adaptive
    estimation over anisotropic Besov spaces or reproductive kernel
    Hilbert spaces. Finally, we provide an empirical evidence that
    aggregation leads to more stable estimators than more standard
    cross-validation or generalized cross-validation methods for the
    selection of the smoothing parameter, when the number of
    observation is small.
  \end{abstract}



\end{frontmatter}

\section{Introduction}
\label{sec:introduction}

\subsection{Motivations}

In this paper, we explore some statistical properties of penalized
empirical risk minimization (PERM) and aggregation procedures in the
regression model. From these properties, we will be able to obtain
results concerning adaptive estimation for several problems. Given a
data set $D_n$, we consider two problems. Let us define the norm
$\norm{g}^2 := \int g(x)^2 P_X(dx)$ where $P_X$ is the law of the
covariates and let $E[\cdot]$ be the expectation w.r.t. the joint law
of $D_n$. The first problem is the problem of estimation of the
regression function $f_0$. Namely, we aim at constructing some
procedure $\bar{f}_n$ satisfying
\begin{equation}
  \label{eq:RateOfConvergence}
  E \|\bar{f}_n - f_0 \|^2 \leq \psi(n)
\end{equation}
where $\psi(n)$, called the {\it rate of convergence}, is a quantity
we wish very small as $n$ increases. To get this kind of inequality,
it is well-known that one has to assume that $f_0$ belongs to a set
with a small complexity (cf., for instance, the "No free Lunch
theorem" in \cite{DGL:96}). This is what we do in
Section~\ref{sec:pena_least_squares} below, where an assumption on the
complexity is considered, see Assumption ($C_\beta$) on the metric
entropy.

However, this kind of ``a priori'' may not be fulfilled. That is why
the second problem, called {\it agnostic learning} has been introduced
(cf. \cite{H:92,KSSH:94} and references therein). For this problem, one is given a set $F$ of
functions. Without any assumption on $f_0$, we want to construct (from
the data) a procedure $\tilde{f}$ which has a risk as close as
possible to the smallest risk over $F$. Namely, we want to obtain {\it
  oracle inequalities}, that is inequalities of the form
\begin{equation*}
  E \| \tilde{f} - f_0 \|^2 \leq C \min_{f\in F} \|f - f_0 \|^2 +
  \phi(n,F),
\end{equation*}
where $C \geq 1$ and $\phi(n,F)$ is called the {\it residue}, which is
the quantity that we want to be small as $n$ increases.  When $F$ is
of finite cardinality $M$, the agnostic problem is called {\it
  aggregation problem} and the residue $\phi(n,F) = \phi(n,M)$ is
called {\it rate of aggregation}. The main difference between the
problems of estimation and aggregation is that we don't need any
assumption on $f_0$ for the second problem. Nevertheless, aggregation
methods have been widely used to construct adaptive procedures for the
estimation problem. That is the reason why we study aggregation
procedures in Section~\ref{sec:ERM_finite} below. We will use these
procedures in Section~\ref{sec:examples} to construct adaptive
estimators in several particular cases, such as adaptive estimation in
reproductive kernel Hilbert spaces (RKHS) or adaptive estimation over
anisotropic Besov spaces.

In Section~\ref{sec:ERM_finite}, we also prove that the ``natural''
aggregation procedure, namely empirical risk minimization (ERM) (or
its penalized version), fails to achieve the optimal rate of
aggregation in this setup. This result motivates the use of an
aggregation procedure instead of the most common ERM. Moreover, we
provide an empirical evidence in Section~\ref{sec:simulations} that
aggregation (with jackknife) is more stable than the classical
cross-validation or generalized cross-validation procedures when the
number of observations and the signal-to-noise ratio are small.

The approach proposed in this paper allows to give rates of
convergence for adaptive estimators over very general function sets,
such as the anisotropic besov space, with very mild assumption on the
law of the covariates: all the results are stated with the sole
assumption that the law of the covariates is compact.

\subsection{The model}
\label{sec:model}

Let $(X, Y), (X_1, Y_1), \ldots, (X_n, Y_n)$, be independent and
identically distributed variables in $\mathbb R^d \times \mathbb
R$. We consider the regression model
\begin{equation}
  \label{eq:model}
  Y = f_0(X) + \sigma \varepsilon,
\end{equation}
where $f_0 : \mathbb R^d \rightarrow \mathbb R$ and $\varepsilon$ is
called noise. To simplify, we assume that the noise level $\sigma$ is
known. We denote by $P$ the probability distribution of $(X,Y)$ and by
$P_X$ the margin distribution in $X$ or \emph{design}, or
\emph{covariates} distribution. We denote by $P^n$ the joint
distribution of the sample
\begin{equation*}
  D_n := [ (X_i, Y_i) \;;\; 1 \leq i \leq n],
\end{equation*}
and by $P_n = P^n[\cdot | X^n]$ where $X^n := (X_1, \ldots, X_n)$, the
joint distribution of the sample $D_n$ conditional on the design $X^n
:= (X_1, \ldots, X_n)$. The expectation w.r.t. $P_n$ is denoted by
$E_n$. The noise $\varepsilon$ is symmetrical and subgaussian
conditionally on $X$. Indeed, we assume that there is $b_\varepsilon >
0$ such that
\begin{equation}
  \label{eq:subgaussian}
  (G1)(b_\varepsilon): \quad E[\exp(t\varepsilon) | X] \leq
  \exp(b_\varepsilon^2t^2/2) \quad \forall t > 0
\end{equation}
which is equivalent (up to an appropriate choice for the constant
$b_\varepsilon$) to
\begin{equation*}
  \nonumber(G2)(b_\varepsilon) : P[\varepsilon > t | X] \leq
  \exp(-t^2/(2b_\varepsilon^2)) \quad \forall t > 0.
\end{equation*}
Assumption~\eqref{eq:subgaussian} is standard in nonparametric
regression, it includes the models of bounded and Gaussian
regression. An important fact, that will be used in the proofs, is
that for $\varepsilon_1,\ldots,\varepsilon_n$ independent and such
that $\varepsilon_i$ satisfies $(G1)(b_i)$ for any $i=1,\ldots,n$, the
random variable $\sum_{i=1}^n a_i \varepsilon_i$ satisfies $(G1)(\sum
a_i^2b_i^2$) for any $a_1,\ldots,a_n \in \R$ and thus the
concentration property $(G2)(\sqrt{2}\sum a_i^2b_i^2$). Other
equivalent definitions of subgaussianity are, when $\varepsilon$ is
symmetrical, to assume that $E[ \exp(\varepsilon^2/b_\varepsilon^2 |
X) ] \leq 2$ for some $b_\varepsilon > 0$, or $(E[ |\varepsilon|^p |
X])^{1/p} \leq b_\varepsilon \sqrt{p}$ for any $p \geq 1$.

Concerning the design, we only assume that $X$ has a compact support,
and without loss of generality we can take its support equal to $[0,
1]^d$. In particular we do not need $P_X$ to be continuous with
respect to the the Lebesgue measure. Note that the problem of adaptive
estimation with such a general multivariate design is not common in
literature. In the so-called ``distribution free nonparametric
estimation'' framework, when we want to obtain convergence rates and
not only the consistency of the estimators, it is, as far as we know,
always assumed that $|Y| \leq L$ a.s. for some constant $L > 0$, see
for instance~\cite{kohler02}, \cite{kohler_krzyzak01a},
\cite{kohler_krzyzak01b}, \cite{kohler00} and~\cite{kerk_picard07},
which is a setting less general than the one considered here.

\begin{remark}
  The results presented here can be extended to subexponential noise,
  that is when $E[ \exp(|\varepsilon| / b_\varepsilon) | X] \leq 2$
  for some $b_\varepsilon > 0$, but it involves complications
  (chaining with an adaptative truncation argument in the proof of
  Theorem~\ref{thm:devia1} below, see for instance~\cite{BLM99}
  or~\cite{van_de_geer00}, among others) that we prefer to skip
  here. 
\end{remark}

\section{PERM over a large function set}
\label{sec:pena_least_squares}

We consider the following problem of estimation: we fix a function
space $\mathcal F$ and we want to recover $f_0$ based on the sample
$D_n$ using the knowledge that $f_0 \in \mathcal F$. The set $\mathcal
F$ is endowed with a seminorm $|\cdot|_{\mathcal F}$. To fix the
ideas, when $d=1$, one can think for instance of the Sobolev space
$\mathcal F = W_2^s$ of functions such that $|f|_{\mathcal F}^2 = \int
f^{(s)}(t)^2 dt < +\infty$, where $s$ is a natural integer and
$f^{(s)}$ is the $s$-th derivative of $f$. In this case, the estimator
described below is the so-called \emph{smoothing spline estimator},
see for instance \cite{wahba90}. Several other examples are given in
Section~\ref{sec:examples} below.

\subsection{Definition of the PERM}

The idea of penalized empirical risk minimization is to make the
balance between the goodness-of-fit of the estimator to the data with
its smoothness. The quantity $|f|_{\mathcal F}$ measures the
smoothness (or ``roughness'') of $f \in \mathcal F$ and the balance is
quantifyied by a parameter $h > 0$.
\begin{definition}[PERM]
  \label{def:perm}
  Let $\lambda = (h, \mathcal F)$ be fixed. We say that $\bar
  f_\lambda$ is a penalized empirical risk minimizer if it minimizes
  \begin{equation}
    \label{eq:pena_least_sq}
    R_n(f)  + \pen_\lambda(f)
  \end{equation}
  over $\mathcal F$, where $\pen_\lambda(f) := h^2 |f|_{\mathcal
    F}^\alpha$ for some $\alpha > 0$ and where
  \begin{equation*}
    R_n(f) := \norm{Y - f}_n^2 = \frac{1}{n} \sum_{i=1}^n (Y_i -
    f(X_i))^2
  \end{equation*}
  is the empirical risk of $f$ over the sample $D_n$.
\end{definition}

The parameter $\alpha$ is a tuning parameter, which can be chosen
depending on the seminorm $|\cdot|_{\mathcal F}$, see the examples in
Section~\ref{sec:examples}. For simplicity, we shall always assume
that a PERM $\bar f_\lambda$ exists, since we can always find $\tilde
f_\lambda$ such that $R_n(\tilde f_\lambda) + \pen_{\lambda}(\tilde
f_\lambda) \leq \inf_{f \in \mathcal F} \{ R_n(f) + \pen_{\lambda}(f)
\} + 1 / n$ which satisfies the same upper bound from
Theorem~\ref{thm:least_sq} (see below) as an hypothetic $\bar
f_\lambda$. However, a minimizer may not be necessarily unique, but
this is not a problem for the theoretical results proposed below. PERM
has been studied in a tremendous number of papers, we only refer to
\cite{van_de_geer00, vdg07}, \cite{massart03} and \cite{kohler02},
which are the closest to the material proposed in this Section.

In Theorem~\ref{thm:least_sq} below we propose a general upper bound
for PERM over a space $\mathcal F$ that satisfies the complexity
Assumption $(C_\beta)$ below. The proof of this upper bound involves a
result concerning the supremum of the empirical process $Z(f) :=
\sigma n^{-1/2} \sum_{i=1}^n f(X_i) \varepsilon_i$ over $f \in
\mathcal F$ which is given in Theorem~\ref{thm:devia1} below.


\subsection{Some definitions and useful tools}

Let $(E, \norm{\cdot})$ be a normed space. For $z \in E$, we denote by
$B(z, \delta)$ the ball centered at $z$ with radius $\delta$. We say
that $\{ z_1, \ldots, z_p \}$ is a $\delta$-cover of some set $A
\subset E$ if
\begin{equation*}
  A \subset \bigcup_{1 \leq i \leq p} B(z_i, \delta).
\end{equation*}
The \emph{$\delta$-covering number} $N(\delta, A, \norm{\cdot})$ is
the minimal size of a $\delta$-cover of~$A$ and
\begin{equation*}
  H(\delta, A, \norm{\cdot}) := \log N(\delta, A, \norm{\cdot})
\end{equation*}
is the \emph{$\delta$-entropy} of $A$. The main assumption in this
section concerns the complexity of the space $\mathcal F$, which is
quantified by a bound on the entropy of its unit ball $B_{\mathcal F}
:= \{ f \in \mathcal F : |f|_{\mathcal F} \leq 1 \}$. We denote for
short $H_\infty(\delta, A) = H(\delta, A, \norm{\cdot}_\infty)$ where
$\norm{f}_\infty := \sup_{x \in [0,1]^d} |f(x)|$. We denote by $C([0,
1]^d)$ the set of continuous functions on $[0, 1]^d$.
\begin{assumption}[$C_\beta$]
  We assume that $\mathcal F \subset C([0, 1]^d)$ and that there is a
  number $\beta \in (0, 2)$ such that for any $\delta > 0$, we have
  \begin{equation}
    H_\infty\big( \delta, B_{\mathcal F} \big)
    \leq D \delta^{-\beta}
  \end{equation}
  where $D > 0$ is independent of $\delta$.
\end{assumption}
This assumption entails that, for any radius $R > 0$, we have
\begin{equation*}
  H_\infty\big( \delta, B_{\mathcal F}(R) \big) \leq D
  \Big(\frac{R}{\delta}\Big)^{\beta}
\end{equation*}
where $B_{\mathcal F}(R) := \{ f \in \mathcal F : |f|_{\mathcal F}
\leq R \}$.
Assumption~$(C_\beta)$ is satisfied by barely all smoothness spaces
considered in nonparametric literature (at least when the smoothness
of the space is large enough compared to the dimension, see
below). The most general space that we consider in this paper and
which satisfies~$(C_\beta)$ is the anisotropic Besov space $B_{p,
  q}^{\bs s}$, where $\bs s = (s_1, \ldots, s_d)$ is a vector of
positive numbers. This space is precisely defined in
Appendix~\ref{sec:appendix_approximation}. Each $s_i$ corresponds to
the smoothness in the direction $e_i$, where $\{ e_1, \ldots, e_d \}$
is the canonical basis of $\mathbb R^d$. The computation of the
entropy of $B_{p, q}^{\bs s}$ can be found in~\cite{triebel06}, we
give more details in Appendix~\ref{sec:appendix_approximation}. If
$\bs {\bar s}$ is the harmonic mean of $\bs s$, namely
  \begin{equation}
    \label{eq:harmonic_mean}
    \frac{1}{\bs {\bar s}} := \frac{1}{d} \sum_{i=1}^d
    \frac{1}{s_i},
  \end{equation}
  then $B_{p, q}^{\bs s}$ satisfies~$(C_\beta)$ with $\beta = d / \bs
  {\bar s}$, given that $\bs {\bar s} > d / s$, which is the usual
  condition to have the embedding $B_{p, q}^{\bs s} \subset C([0,
  1]^d)$.

\begin{remark}
  Under the restriction $\beta \in (0, 2)$, the Dudley's entropy
  integral satisfies
  \begin{equation*}
    \int_0^{ {\rm diam}( B_{\mathcal F}, \|\cdot\|_\infty)}
    \sqrt{H_\infty(\delta, B_{\mathcal F})} d\delta<\infty,
  \end{equation*}
  where $\text{diam}(B_{\cF},\|\cdot\|_\infty)$ is the
  $L_\infty$-diameter of $B_{\mathcal F}$.  This is a standard
  assumption coming from empirical process theory. It is related to
  the so-called chaining argument, that we use in the proof of
  Theorem~\ref{thm:devia1}. However, in order to consider a larger
  space of functions $\mathcal F$, we could think of function spaces
  with a complexity $\beta \geq 2$. In this case, using a slightly
  different chaining argument (cf. \cite{vdVW:96}), the quantity
  appearing in the upper bound of some subgaussian process is of the
  type $\int_{c/\sqrt{n}}^{\text{diam}(B_{ \cF},\|\cdot\|_\infty)}
  \sqrt{H_\infty(\delta, B_{\mathcal F})} d\delta$ which converges
  whatever $\beta$ is. However, such considerations are beyond the
  scope of the paper and are to be considered in a future work.
\end{remark}






\subsection{About the supremum of the process $Z(\cdot)$}
\label{sec:process_Z0}

The beginning of the proof of Theorem~\ref{thm:least_sq} is, as usual
with the proof of upper bounds for $M$-estimators, based on an
inequality that links the empirical norm of estimation and the
empirical process of the model. This idea goes back to key
papers~\cite{vandegeer90} and \cite{birge_massart93}, see
also~\cite{van_de_geer00, vdg07} and \cite{massart03} for a detailed
presentation. In regression, it writes, if $\bar f$ is a PERM and if
$f_0 \in \mathcal F$:
\begin{align*}
  \norm{\bar f - f_0}_n^2 + \pen(\bar f) &\leq \frac{2}{\sqrt{n}}
  Z_n(\bar f - f_0) + \pen(f_0),
\end{align*}
where
\begin{equation}
  \label{eq:Z_n_def}
  Z_n(f) := \frac{\sigma}{\sqrt{n}} \sum_{i=1}^n f(X_i) \varepsilon_i.
\end{equation}
This inequality explains why the next Theorem~\ref{thm:devia1} is the
main ingredient of the proof of Theorem~\ref{thm:least_sq}
below. Then, an important remark is that~\eqref{eq:subgaussian}
entails
\begin{equation}
  \label{eq:deviaZnf}
  P_n[Z_n(f) > z] \leq \exp\Big( \frac{-z^2}{2 b^2 \norm{f}_n^2}
  \Big)
\end{equation}
for any fixed $f$, $z > 0$ and $n \geq 1$, where $\norm{f}_n^2 :=
n^{-1} \sum_{i=1}^n f(X_i)^2$ and where we take for short $b := \sigma
b_\varepsilon$. This deviation inequality is at the core of the proof
of Theorem~\ref{thm:devia1} below. Let us introduce the
\emph{empirical ball} $B_n(f_0, \delta) := \{ f : \norm{f - f_0}_n
\leq \delta \}$ and let us recall that $P_n := P^n[\cdot | X^n]$ is
the joint law of the sample $D_n$ conditionally to the design $X^n =
(X_1, \ldots, X_n)$.

\begin{theorem}
  \label{thm:devia1}
  Let $Z_n(\cdot)$ be the empirical process~\eqref{eq:Z_n_def} and
  assume that $(\mathcal F, |\cdot|_{\mathcal F})$ satisfies
  $(C_\beta)$. Then\textup, if $f_0 \in \mathcal F$\textup, we can
  find constants $z_1 > 0$ and $D_1 > 0$ such that\textup:
  \begin{align}
    \label{eq:deviaZ_n}
    P_n \Big[ \sup_{f \in \mathcal F \cap B_n(f_0, \delta)} \frac{
      Z_n(f - f_0) }{\norm{f - f_0}_n^{1 - \beta / 2} (1 +
      |f|_{\mathcal F})^{\beta / 2} } > z \Big] \leq \exp( - D_1 z^2
    \delta^{-\beta} )
  \end{align}
  for any $\delta > 0$ and $z \geq z_1$ \textup(we recall that $\beta
  \in (0, 2)$\textup).
\end{theorem}

The proof of this Theorem is given is
Section~\ref{sec:proof_main_results}, it uses techniques from
empirical process theory such as peeling and chaining. It is a uniform
version of~\eqref{eq:deviaZnf}, localized around $f_0$ (for the
empirical norm). In this theorem, we use the ``weighting trick'' that
was introduced in~\cite{vandegeer90, van_de_geer00}: we divide
$Z_n(\cdot)$ by $\norm{f - f_0}_n$ and $|f|_{\mathcal F}$ in order to
counterpart, respectively, the variance of $Z_n(\cdot)$ and the
massiveness of the class $\mathcal F$. This renormalization of the
empirical process is also at the core of the proof of
Theorem~\ref{thm:least_sq}.




\subsection{Upper bound for the PERM}

Theorem~\ref{thm:least_sq} below provides an upper bound for the mean
integrated squared error (MISE) of the PERM, both for integration
w.r.t. the empirical norm $\norm{f}_n^2 = n^{-1} \sum_{i=1}^n
f(X_i)^2$ and the norm $\norm{f}^2 := \int f(x)^2 P_X(dx)$.

\begin{theorem}
  \label{thm:least_sq}
  Let $\mathcal F$ be a space of functions satisfying $(C_\beta)$.
  Let $\lambda = (h, \mathcal F)$ and $\bar f_{\lambda}$ be a PERM
  given by~\eqref{eq:pena_least_sq}, where $h$ satisfies
  \begin{equation}
    \label{eq:bandwidth}
    h = a n^{-1 / (2 + \beta)}
  \end{equation}
  for some constant $a > 0$ and where $\alpha > 2\beta / (\beta +
  2)$. If $f_0 \in \mathcal F$, we have\textup:
  \begin{equation*}
    E_n \norm{\bar f_{\lambda} - f_0}_n^2 \leq C_1(1 + |f_0|_{\mathcal
      F}^\alpha) n^{-2 / (2 + \beta)}
  \end{equation*}
  for $n$ large enough, where $C_1$ is a fixed constant depending on
  $a$, $\beta$, $\alpha$ and $b$. If we assume further that
  $\norm{\bar f_\lambda - f_0}_\infty \leq Q$ a.s. for some constant
  $Q > 0$, we have
  \begin{equation*}
    E^n \norm{\bar f_\lambda - f_0}^2 \leq C_2 (1 + |f_0|_{\mathcal
      F}^\alpha ) n^{-2 / (2 + \beta)}
  \end{equation*}
  for $n$ large enough, where $C_2$ is a fixed constant depending on
  $C_1$ and $Q$.
\end{theorem}


\begin{remark}
  Theorem~\ref{thm:least_sq} holds if we truncate $\bar f_\lambda$ by
  some constant $Q$ such that $\norm{f_0}_\infty \leq Q$. Such a
  truncation cannot be avoided in such a general regression
  setting. Indeed, the PERM is, without truncation, in general non
  consistent, see the example from Problem~20.4, p.~430
  in~\cite{kohler02}.
\end{remark}

\begin{remark}
  Theorem~\ref{thm:least_sq} holds for any design law $P_X$, even for
  the degenerate case where $P_X = \delta_x$ for some fixed point $x
  \in [0,1]^d$, where $\delta$ is the Dirac probability measure. Of
  course, in this case, the rate $n^{-2 / (2 + \beta)}$ becomes
  suboptimal, since the estimation problem with such a $P_X$ is no
  more ``truly nonparametric''. Indeed, for a discrete $P_X$ with
  finite support, it is proved in~\cite{hamers_kohler04} that the
  optimal rate is the parametric rate $1/n$ using a local averaging
  estimator.
\end{remark}


\subsection{About the smoothing parameter $h$}
\label{sec:about_h}

It is well-known that in practice, the choice of the parameter $h$ is
of first importance. From the theoretical point of view, in order to
make $\bar f_\lambda$ rate-optimal, $h$ must equal in order to a
quantity involving the complexity of $\mathcal F$: see
condition~\eqref{eq:bandwidth} on the bandwidth and the
Assumption~$(C_\beta)$. This problem is commonplace in nonparametric
statistics. Indeed, the role of the penalty
in~\eqref{eq:pena_least_sq} is to make the balance with the
massiveness of the space $\mathcal F$. Without this penalty, or if $h$
is too small, $\bar f_{\lambda}$ roughly interpolates the data, which
is not suitable when the aim is denoising (this phenomenon is called
\emph{overfitting}).

Of course, the complexity parameter $\beta$ is unknown to the
statistician, and even worse, it does not necessarily make sense in
practice. So, several procedures are proposed to select $h$ based on
the data. The most popular are the leave-one-out cross validation (CV)
and the simpler generalized cross validation (GCV), which is often
used with smoothing spline estimators because of its computational
simplicity, see~\cite{wahba90} among others. Such methods are known to
provide good results in most cases. However, there is, as far as we
know, no convergence rates results for estimators based on CV or GCV
selection of smoothing parameters. In Section~\ref{sec:examples}
below, we propose an alternative approach. Indeed, instead of
selecting one particular $h$, we mix several estimators computed for
different $h$ in some grid using an aggregation algorithm. This
aggregation algorithm is described in Section~\ref{sec:ERM_finite}. We
show that this approach allows to construct adaptive estimators with
optimal rates of convergence in several particular cases, see
Section~\ref{sec:examples}. Moreover, we prove empirically in
Section~\ref{sec:simulations} that the aggregation approach is more
stable than CV or GCV when the number of observations is small.

\section{PERM and aggregation over a finite set of functions}
\label{sec:ERM_finite}

Let us fix a set $F(\Lambda) := \{ f_\lambda : \lambda \in \Lambda \}$
of arbitrary functions, and denote by $M = |\Lambda|$ its
cardinality. 

\subsection{Suboptimality of PERM over a finite set}

In this section, we prove that minimizing the empirical risk
$R_n(\cdot)$ (or a penalized version) on $F(\Lambda)$ is a suboptimal
aggregation procedure in the sense of~\cite{tsy:03}. According to
\cite{tsy:03}, the optimal rate of aggregation in the gaussian
regression model is $(\log M) /n$. This means that it is the minimum
price one has to pay in order to mimic the best function among a class
of $M$ functions with $n$ observations. This rate is achieved by the
aggregate with cumulative exponential weights, see~\cite{catbook:01}
and~\cite{jrt:06}.
In Theorem~\ref{TheoWeaknessERMRegression} below, we prove that the
usual PERM procedure cannot achieve this rate and thus, that it is
suboptimal compared to the aggregation methods with exponential
weights. The lower bounds for aggregation methods appearing in the
literature (see~\cite{tsy:03, jrt:06, LecJMLR:06}) are usually based
on minimax theory arguments. The one considered here is based on
geometric considerations, and involves an explicit example that makes
the PERM fail. For that, we consider the Gaussian regression model
with uniform design.
\begin{assumption}[G]
  Assume that $\varepsilon$ is standard Gaussian and that $X$ is
  univariate and uniformly distributed on $[0, 1]$.
\end{assumption}
\begin{theorem}
  \label{TheoWeaknessERMRegression}
  Let $M \geq 2$ be an integer and assume that \textup{(G)} holds. 
  We can find a regression function $f_0$ and a family $F(\Lambda)$ of
  cardinality $M$ such that, if one considers a penalization
  satisfying $|\pen(f)| \leq C \sqrt{(\log M)/n}, \forall f \in
  F(\Lambda)$ with $0\leq C <\sigma (24\sqrt{2}c^*)^{-1}$ \textup($c^*$ is
  an absolute constant from the Sudakov minorization, see
  Theorem~\ref{TheoSudakov} in
  Appendix~\ref{sec:appendix_proba}\textup), the PERM procedure
  defined by
  \begin{equation*}
    \tilde{f}_n \in \argmin_{f \in F(\Lambda)}( R_n(f) + \pen(f))
  \end{equation*}
  satisfies
  \begin{equation*}
    E^n \| \tilde{f}_n - f_0 \|^2 \geq \min_{f \in
      F(\Lambda)} \| f - f_0 \|^2 + C_3 \sqrt{\frac{\log
        M}{n}}
  \end{equation*}
  for any integer $n \geq 1$ and $M\geq M_0(\sigma)$ such that $n^{-1}
  \log[(M-1)(M-2)] \leq 1/4$ where $C_3$ is an absolute constant.
\end{theorem}
This result tells that, in some particular cases, the PERM cannot
mimic the best element in a class of cardinality $M$ faster than
$((\log M)/n)^{1/2}$. This rate is very far from the optimal one
$(\log M)/n$.

Let $F(\Lambda)$ be the set that we consider in the proof of
Theorem~\ref{TheoWeaknessERMRegression} (see
Section~\ref{sec:proof_main_results} below), and take $\pen(f) = 0$.
Using Monte-Carlo (we do $5000$ loops), we compute the excess risk $E
\| \tilde{f}_n - f_0 \|^2 - \min_{f \in F(\Lambda)} \| f - f_0 \|^2$
of the ERM. In Figure~\ref{fig:subERM} below, we compare the excess
risk and the bound $((\log M) / n)^{1/2}$ for several values of $M$
and $n$. It turns out that, for this set $F(\Lambda)$, the lower bound
$((\log M) / n)^{1/2}$ is indeed accurate for the excess
risk. Actually, by using the classical symmetrization argument and the
Dudley's entropy integral, it is easy to obtain an upper bound for the
excess risk of the ERM of the order of $((\log M) / n)^{1/2}$ for any
class $F(\Lambda)$ of cardinality $M$.

\begin{figure}[htbp]
  \centering
  \includegraphics[width=4.3cm]{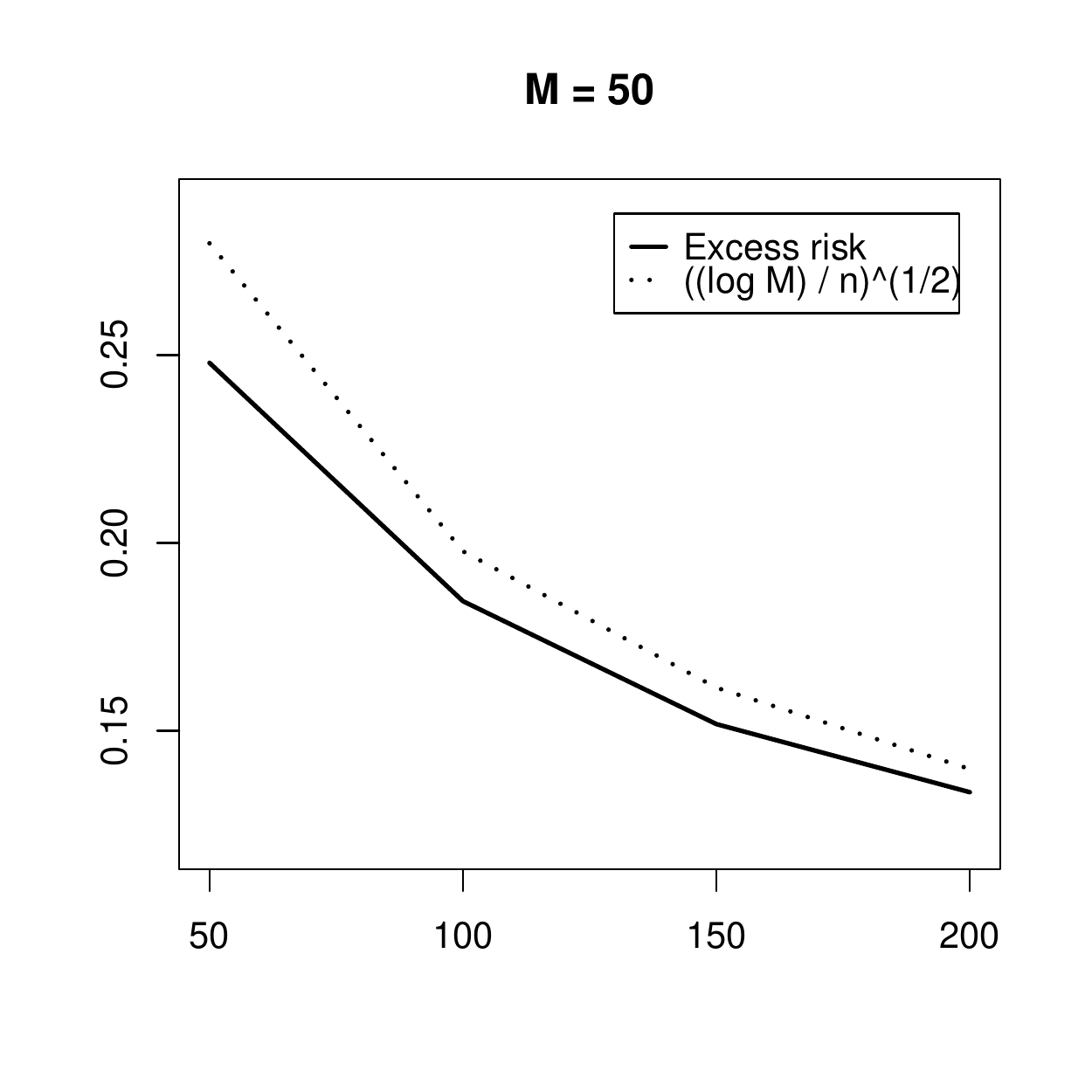}%
  \includegraphics[width=4.3cm]{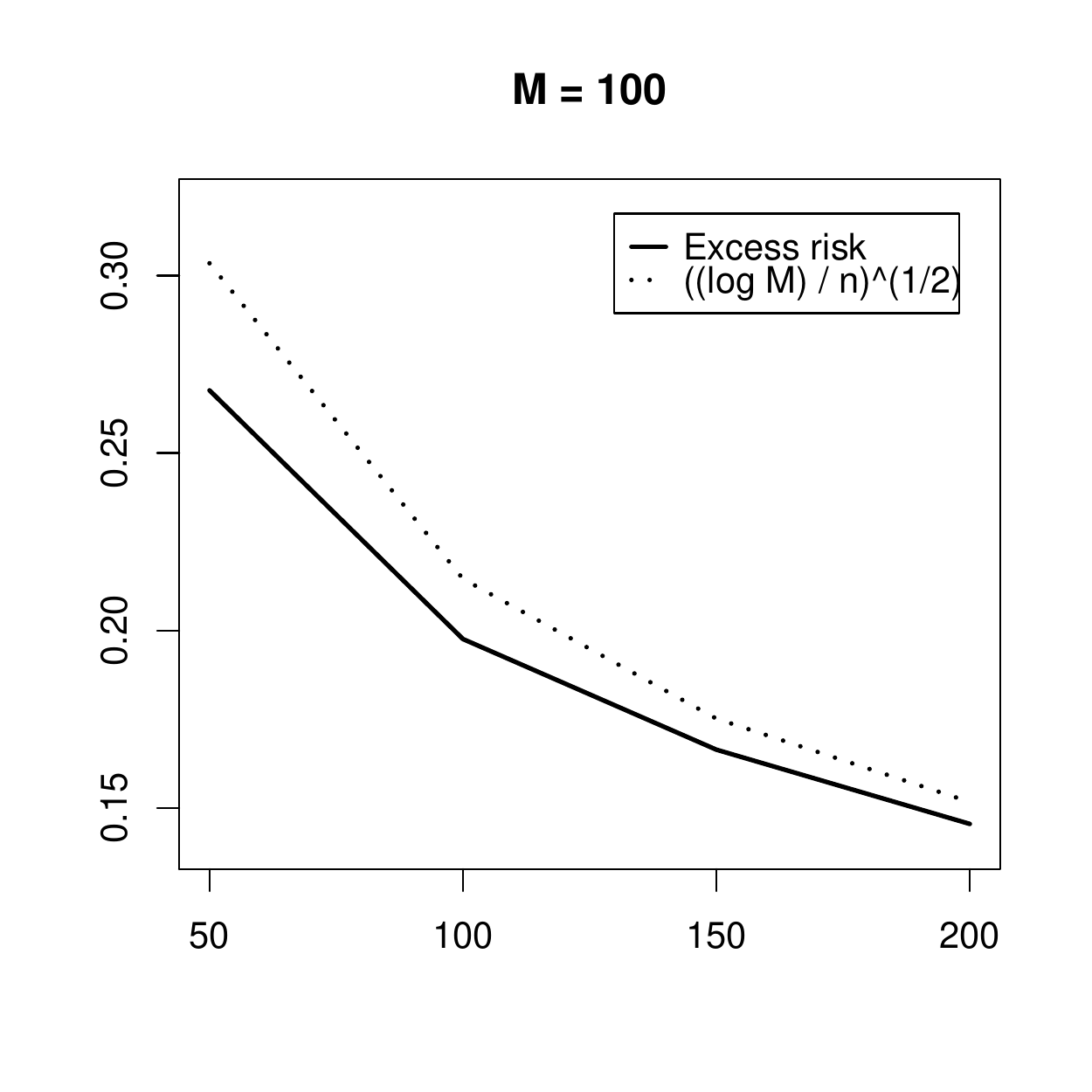}%
  \includegraphics[width=4.3cm]{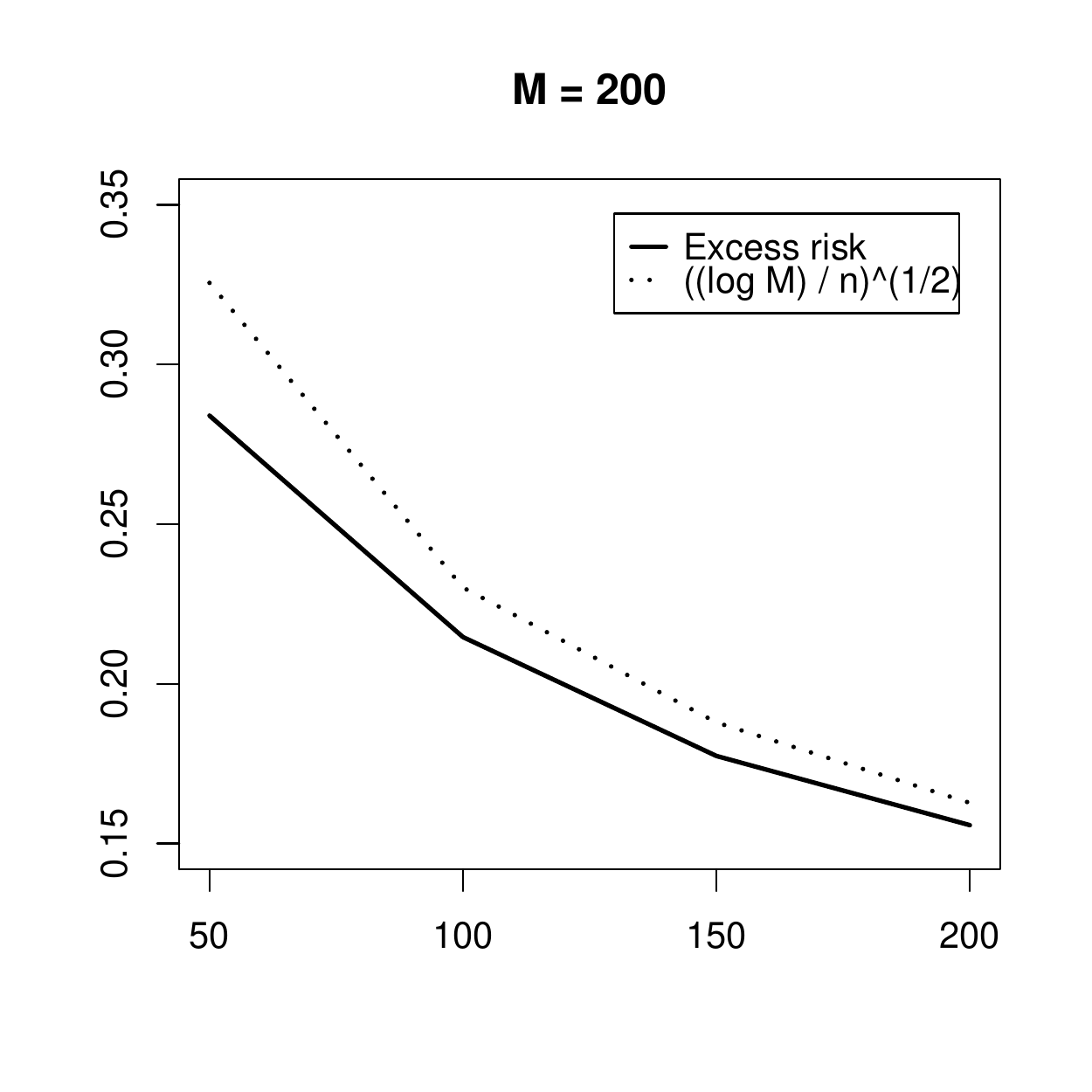}%
  \caption{The excess risk of the ERM compared to $((\log M) /
    n)^{1/2}$ for several values of $M$ and $n$
    \textup($x$-axis\textup)}
  \label{fig:subERM}
\end{figure}

\subsection{Aggregation}
\label{sec:aggregation}

For each $ f_\lambda \in F(\Lambda)$, we compute a weight $\theta(
f_\lambda) \in [0,1]$ such that $\sum_{\lambda \in \Lambda} \theta(
f_{\lambda}) = 1$. These weights give a level of significance to each
$ f_\lambda \in F(\Lambda)$.  The aggregated estimator is then the
convex combination
\begin{equation}
  \label{eq:aggregate}
  \hat {\mathsf f} := \sum_{\lambda \in \Lambda} \theta(f_\lambda)
  f_\lambda,
\end{equation}
where the weight of $f \in F(\Lambda)$ is given by
\begin{equation}
  \label{eq:weights}
  \theta(f) := \frac{\exp\big( - n R_{n}(f) / T
    \big)}{\sum_{\lambda \in \Lambda} \exp\big(-n R_{n}(
    f_\lambda)/T \big) },
\end{equation}
where $T > 0$ is the so-called \emph{temperature} parameter and where
$R_n(f)$ is the empirical risk of $f$. This aggregation algorithm
(with ``Gibbs'' or ``exponential'' weights) can also be found for
instance in~\cite{catbook:01, leung_barron06, juditsky_etal05,
  juditsky_nazin05, yang:00, yang04, LecAoS:07}. See
also~\cite{gaiffas_lecue07} for adaptation by aggregation in a
semiparametric model.

The next theorem is an oracle inequality for the aggregation
method~\eqref{eq:weights}. It will be useful to derive the adaptive
upper bounds stated in Section~\ref{sec:examples} below.
\begin{theorem}
  \label{thm:oracle}
  Assume that for any $f \in F(\Lambda)$, we have $\norm{f -
    f_0}_\infty \leq Q$ for some $Q > 0$. For any $a > 0$, the
  aggregation method~\eqref{eq:weights} satisfies
  \begin{equation*}
    E^n \norm{\hat {\mathsf f} - f_0}^2 \leq (1+ a) \min_{f \in
      F(\Lambda)} \norm{f - f_0}^2 + (C + T) \frac{(\log
      n)^{1/2} \log M}{n},
  \end{equation*}
  where $C$ is a constant depending on $a, Q$ and $\sigma$.
\end{theorem}
When $T$ is too large, the weights~\eqref{eq:weights} are close to the
uniform law over the set of weak estimators, and of course, the
resulting aggregate is inaccurate. When $T$ is too small, one weight
is close to $1$, and the others close to $0$: in this situation, the
aggregate does barely the same job as the ERM procedure. This is not
suitable since Theorem~\ref{TheoWeaknessERMRegression} told us that
ERM is suboptimal. Hence, $T$ realize a tradeoff between the ERM and the
uniform weights procedure.
It can be simply chosen by minimization of the empirical risk. We know
empirically that it provides good results, see~\cite{gaiffas_lecue07}.
Namely, we select the temperature
\begin{equation}
  \label{Tslection}
  \hat T := \argmin_{T \in \mathcal T} \sum_{i=1}^n \big( Y_i - \hat
  {\mathsf f}^{(T)} (X_i) \big)^2,
\end{equation}
where $\hat {\mathsf f}^{(T)}$ is the aggregated
estimator~\eqref{eq:aggregate} with temperature $T$ and where
$\mathcal T$ is some set of temperatures. This is what we do in the
empirical study conducted in Section~\ref{sec:simulations}.

\section{Examples of adaptive results}
\label{sec:examples}


In this section, we construct adaptive estimators for several
regression problems using the tools from
Section~\ref{sec:pena_least_squares} and~\ref{sec:ERM_finite}. This
involves, as usual with algorithms coming from statistical learning
theory, a split of the sample into two parts (an exception can be
found in~\cite{leung_barron06}). The main steps of the construction of
adaptive estimators given in this section are:
\begin{enumerate}
\item split, at random, the whole sample $D_n$ into a \emph{training
    sample}
\begin{equation*}
  D_m := [(X_i, Y_i) : 1 \leq i \leq m],
\end{equation*}
where $m < n$, and a \emph{learning sample}
\begin{equation*}
  D_{(m)} := [(X_i, Y_i) : m + 1 \leq i \leq n];
\end{equation*}
\item choose a set $\Lambda$ of parameters and compute, using the
  training sample $D_m$, the corresponding class $F(\Lambda) = \{ \bar
  f_\lambda : \lambda \in \Lambda \}$ of PERM (see
  Definition~\ref{def:perm} in
  Section~\ref{sec:pena_least_squares}). Each $\Lambda$ depends on the
  considered problem of adaptive estimation, see below;
\item using the learning sample $D_{(m)}$, compute the aggregation
  weights and the aggregated estimator $\hat {\mathsf f}_n$,
  respectively given by Equations~\eqref{eq:weights}
  and~\eqref{eq:aggregate}.
\end{enumerate}

Then, using Theorem~\ref{thm:least_sq} (see
Section~\ref{sec:pena_least_squares}) and Theorem~\ref{sec:ERM_finite}
(see Section~\ref{sec:ERM_finite}), we will derive adaptive upper
bounds for estimators $\hat {\mathsf f}_n$ constructed in this
way. Throughout the section, we shall assume the following.

\begin{assumption}[Split size]
  Let $\ell$ be learning sample size, so that $\ell + m = n$. We shall
  assume from now on, to simplify the presentation, that $\ell$ is a
  fraction of $n$, typically $n/2$ or $n/4$.
\end{assumption}

\subsection{About the split, jackknife}
\label{sec:jackknife}

The behavior of the aggregate $\hat {\mathsf f}_n$ can depend strongly
on the split selected in Step~1, in particular when the number of
observations is small. Hence, a good strategy is to jackknife: repeat,
say, $J$ times Steps 1--3 to obtain aggregates $\{ \hat {\mathsf
  f}_n^{(1)}, \ldots, \hat {\mathsf f}_n^{(J)} \}$, and compute the
mean:
\begin{equation*}
  \hat {\mathsf f}_n := \frac{1}{J} \sum_{j=1}^J \hat {\mathsf
    f}_n^{(j)}.
\end{equation*}
This jackknifed estimator provides better results than a single
aggregate, see Section~\ref{sec:simulations} for an empirical study,
where we show also that it gives more stable estimators than the ones
involving cross-validation of generalized cross-validation. By
convexity of $f \mapsto \norm{f - f_0}^2$, the jackknifed estimator
satisfies the same upper bounds as a single aggregate: each of the
adaptive upper bounds stated below also holds when we use the
jackknife.

For the set of weak estimators considered in this paper, the split of
the data is not a theoretical artefact. Indeed, if one skips Step~1
(compute $F(\Lambda)$ and $\hat {\mathsf f}_n$ using the whole sample
$D_n$), then $\hat {\mathsf f}_n$ has a very poor performance. An
empirical illustration of this phenomenon is given in
Figure~\ref{fig:split_effect}. Herein, we show the aggregation
weights~\eqref{eq:weights} when the data is splitted and when it is
not splitted. We consider an univariate design and cubic smoothing
splines. Namely, we compute the set $F(\Lambda)$ of PERM
(see~\eqref{eq:pena_least_sq}) with $\mathcal F = \{ f \in L^2([0, 1])
: \int f^{(2)}(t) dt < +\infty \}$ and penalty $\pen(f) = h^2 \int
f^{(2)}(t) dt$, where $f^{(2)}$ stands for the second derivative of
$f$. We do that for several smoothing parameters $h$ in a grid $H$, so
that $\Lambda := \{ (h, \mathcal F) : h \in H \}$. We used the
\texttt{smooth.spline} routine in the \texttt{R} software to compute
$F(\Lambda)$.
\begin{figure}[htbp]
  \centering
  \includegraphics[width=6cm]{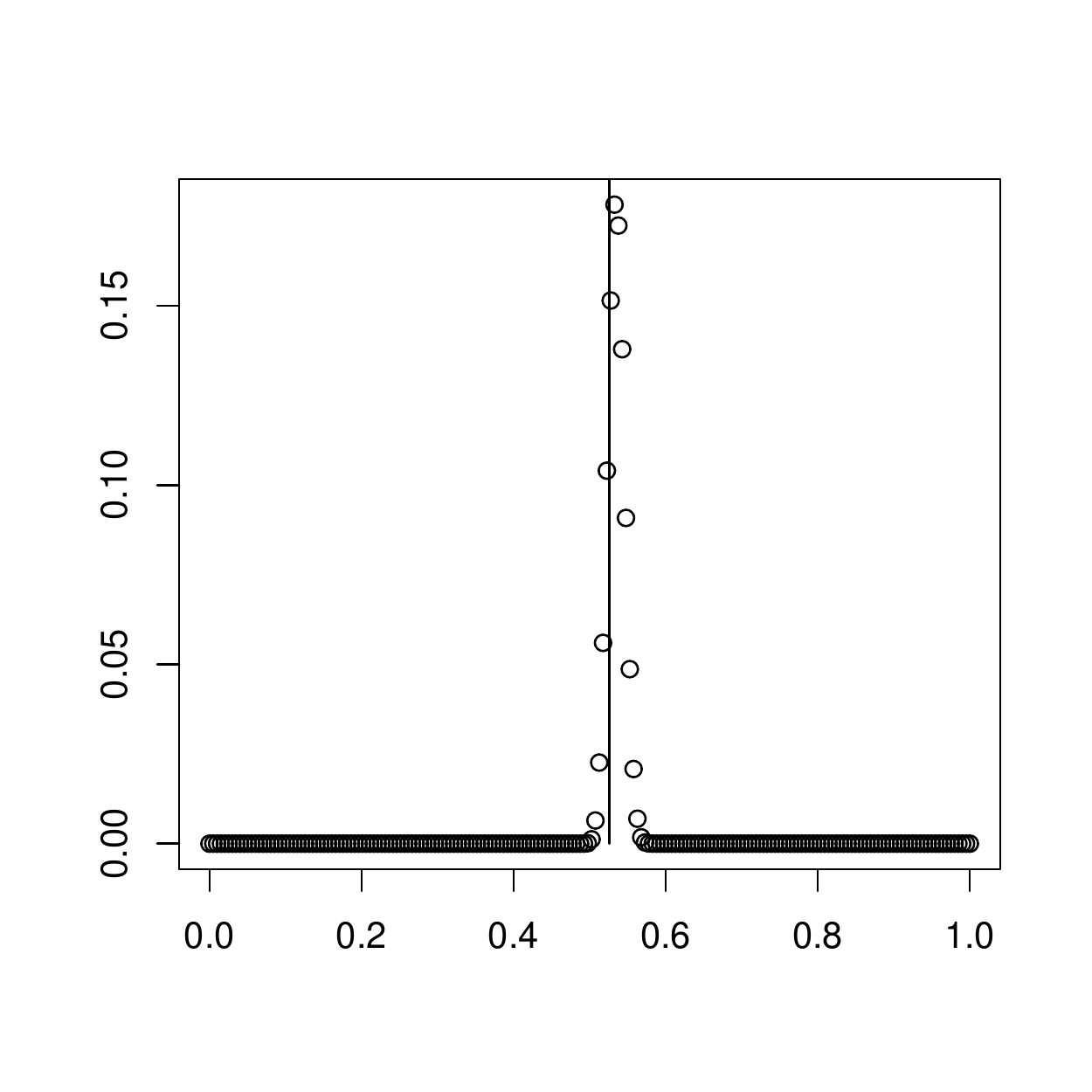}%
  \includegraphics[width=6cm]{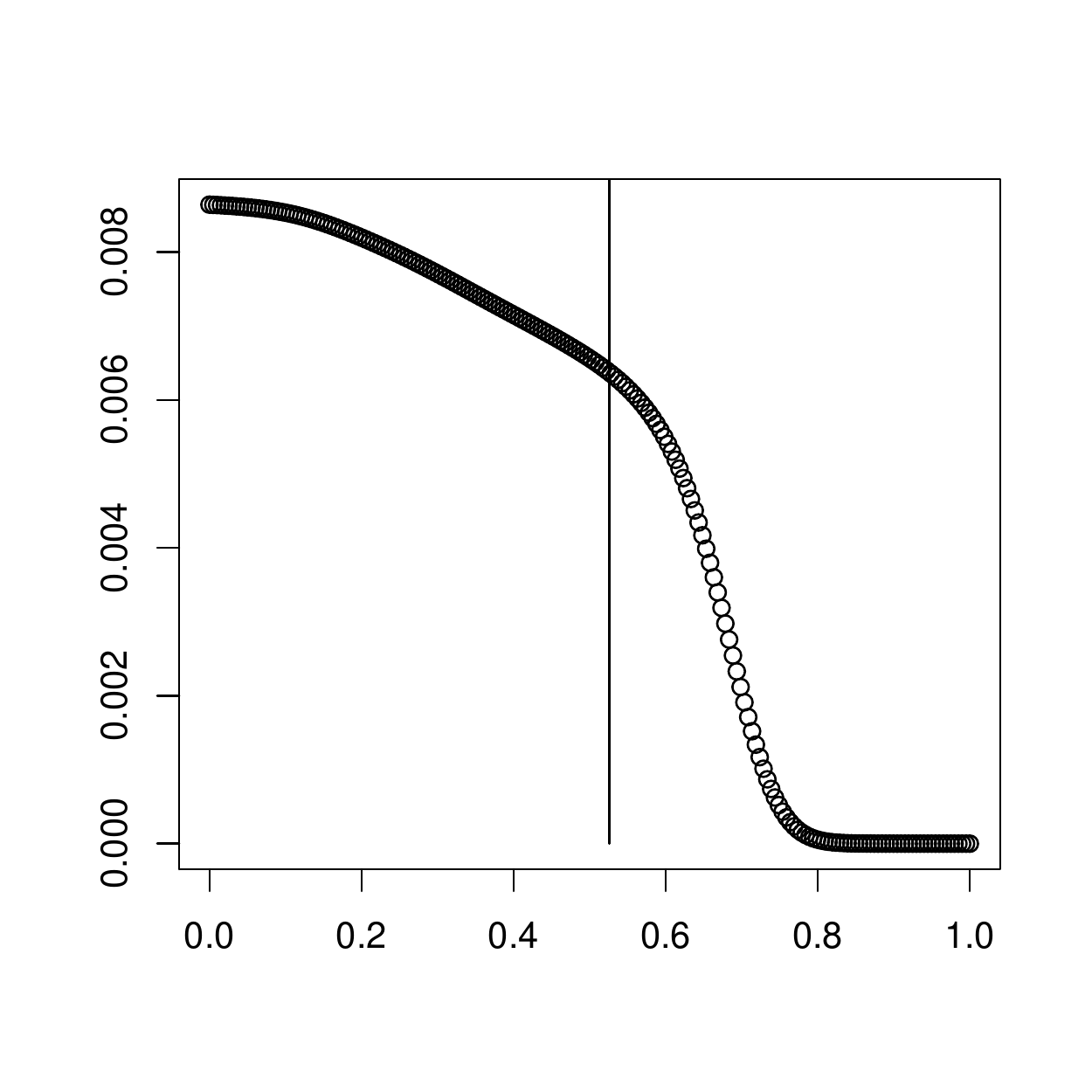}%
  \caption{Aggregation weights with split \textup(left\textup) and
    without split \textup(right\textup) and smoothing parameter
    obtained by cross-validation \textup(vertical line\textup)}
  \label{fig:split_effect}
\end{figure}
In Figure~\ref{fig:split_effect}, the x-axis is related to the value
of $h$: it is the value of the parameter \texttt{spar} from the
\texttt{smooth.spline} routine. The vertical line is the value of
\texttt{spar} selected by cross-validation. The conclusion from
Figure~\ref{fig:split_effect} is that, when the data is not splitted,
an overfitting phenomenon occurs: the aggregation algorithm does not
work, since it does not concentrate around a value of
\texttt{spar}. Of course, the resulting aggregated estimator has a
very poor performance.

\subsection{How to derive the adaptive upper bounds}
\label{sec:derive_adaptive}

In every examples considered below, the scheme to derive adaptive
upper bounds is as follows. Say that $(\mathcal F_\beta : \beta \in
B)$ is a set of embedded functions classes ($\mathcal F_\beta \subset
\mathcal F_{\beta'}$ if $\beta < \beta'$) where each $\mathcal
F_\beta$ satisfy Assumption~$(C_\beta)$. Let $B_n$ be an appropriate
discretization of $B$. Let $\hat {\mathsf f}_n$ be the aggregated
estimator obtained using Steps~1--3 (see the beginning of the
section), with parameter $\Lambda = \Lambda_n = \{ (n^{-2 / (2 +
  \beta)}, \mathcal F_\beta) : \beta \in B_n \}$ and let $M_n$ be the
cardinality of $F(\Lambda_n)$. Let $E^{m}$ and $E^{(m)}$ be the
expectations with respect to, repectively, the joint laws of $D_m$ and
$D_{(m)}$, so that, by independence, we have $E^n[\cdot] =
E^m[E^{(m)}[\cdot]]$. Let $f_0 \in \mathcal F_{\beta_0}$ for some
$\beta_0 \in B$. Using Theorem~\ref{thm:oracle}, we have
\begin{align*}
  E^{(m)} \norm{\hat {\mathsf f}_n - f_0}^2 &\leq C \min_{f \in
    F(\Lambda_n)} \norm{f - f_0}^2 + \frac{C (\log
    n)^{1/2} \log M_n}{n} \\
  & \leq C \norm{\bar f_{\lambda_n} - f_0}^2 + \frac{C (\log n)^{1/2}
    \log M_n}{n},
\end{align*}
where $\lambda_n = (n^{-2 / (2 + \beta_n)}, \mathcal F_{\beta_n})$,
with $\beta_n \in B_n$ chosen such that $\mathcal F_{\beta_0} \subset
\mathcal F_{\beta_n}$ and $n^{-2 / (2 + \beta_n)} \leq C_1 n^{-2 / (2
  + \beta_0)}$. Then, integrating w.r.t. to $E^{m}$ and using
Theorem~\ref{thm:least_sq}, we have, if $M_n$ is no more than a power
of $n$:
\begin{align*}
  E^n \norm{\hat {\mathsf f}_n - f_0}^2 &\leq C E^m \norm{\bar
    f_{\lambda_n} - f_0}^2 + o(n^{-2 / (2 + \beta_0)}) \\
  & \leq C_2 n^{-2 / (2 + \beta_n)} + o(n^{-2 / (2 + \beta_0)}) \leq
  C_3 n^{-2 / (2 + \beta_0)}.
\end{align*}
This prove that, if $f_0 \in \mathcal F_{\beta_0}$ for some $\beta_0
\in B$, we have $E^n \norm{\hat {\mathsf f}_n - f_0}^2 \leq C_3 n^{-2
  / (2 + \beta_0)}$, thus $\hat {\mathsf f}_n$ is indeed adaptive over
$(\mathcal F_\beta : \beta \in B)$.

\subsection{Sobolev spaces, spline estimators}
\label{sec:sobolev_spaces}

When $\mathcal F$ is a Sobolev space, the
PERM~\eqref{eq:pena_least_sq} with $\alpha = 2$ is a very popular
smoothing technique: see, among others, \cite{wahba90} and
\cite{green_silverman94}. The most simple example is when $d=1$ and
\begin{equation*}
  \mathcal F = W_2^s([0, 1]) := \Big\{ f \in L^2([0, 1]) :
  |f|_{W_2^s}^2 := \int_0^1 f^{(s)}(t)^2 dt < \infty \Big\},
\end{equation*}
where $s$ is some natural integer and $f^{(s)}$ stands for the $s$-th
derivative of $f$. In this case, the PERM is called a \emph{smoothing
  spline}, since in this situation the unique minimizer
of~\eqref{eq:pena_least_sq} is a spline, see for
instance~\cite{wahba90} or~\cite{kohler02}. When $s = 2$ (cubic
splines), the routine \texttt{smooth.spline} from the \texttt{R}
software (and for other softwares as well) neatly computes the
solution to~\eqref{eq:pena_least_sq} using the B-spline basis, and
chooses the parameter $h$ via generalized cross-validation (GCV). 

The $d$-dimensional case is easily understood with the definition of
$W_2^s([0, 1]^d)$ as the space of functions $f \in L^2([0, 1]^d)$ with
all derivatives of total order $s$ in $L^2([0,1]^d)$. Namely,
\begin{equation*}
  W_2^s([0, 1]^d) := \Big\{ f \in L^2([0, 1]^d) :
  |f|_{W_2^s([0, 1]^d)}^2 < \infty \Big\},
\end{equation*}
where
\begin{equation}
  \label{eq:usual_roughness}
  |f|_{W_2^s([0, 1]^d)}^2 := \sum_{\mathbf k \in \mathbb N_0^d :
    |\mathbf k| = s} \frac{s
    !}{\mathbf k !} \int_{[0,1]^d} ( D_{\mathbf k} f(x) )^2 dx,
\end{equation}
where for $\mathbf k = (k_1, \ldots, k_d)$ we use the notations
$\mathbf k ! := \prod_{i=1}^d k_i !$ and $|\mathbf k| := \sum_{i=1}^d
k_i$ and where $D_{\mathbf k}$ is the differential operator
$\partial^s / (\partial^{k_1} \cdots \partial^{k_d})$. When $d > 1$,
the PERM for the choice $\mathcal F = W_2^s([0, 1]^d)$ is called a
\emph{thin plate spline}, see again for instance~\cite{wahba90}
or~\cite{kohler02}, where the practical computation of such PERM is
explained in details. The usual assumption $s > d / 2$ gives the
embedding $W_s([0, 1]^d) \subset C[0, 1]^d$ and that
Assumption~$(C_\beta)$ holds, see~\cite{birman_solomjak67}. The
situation where $s$ is not an integer is a particular case of what we
do in Section~\ref{sec:anisotropic_besov} below. The case where
$\mathcal F$ is a Sobolev space is actually a particular case of both
the next sections. Indeed, it is well known (see~\cite{wahba90} for
instance) that a Sobolev space is a Reproductive Kernel Hilbert Space
(RKHS) for an appropriate kernel choice, and that it is also a Besov
space $B_{2, 2}^s$.



\subsection{Reproductive Kernel Hilbert Spaces}
\label{sec:RKHS}

Reproductive Kernel Hilbert Spaces (cf.~\cite{aronszajn50}), RKHS for
short, provide a unified context for regularization in a wide variety
of statistical model. Computational properties of estimators obtained
by minimization of a functional onto a RKHS make these functions space
very useful for statisticians. In this short section, we briefly
recall some definitions and computational properties of RKHS.

Let $\cX$ be an abstract space (in this paper, we take
$\cX=[0,1]^d$). We say that $K:\cX\times\cX\longmapsto\mathbb{R}$ is a
{\it reproducing kernel}, RK for short, if for any integer $p$ and any
points $x_1,\ldots,x_p$ in $\cX$, the matrix $(K(x_i,x_j))_{1\leq
  i,j\leq p}$ is symmetric positive definite. Let $K$ be a RK. The
Hilbert space associated with $K$, called {\it Reproducing Kernel
  Hilbert Space} and denoted by $\cH_K$, is the completion of the
space of all the finite linear combination $\sum_j a_j K(x_j,\cdot)$
endowed with the inner product $\prodsca{\sum_j a_j
  K(x_j,\cdot)}{\sum_k b_k K(y_k,\cdot)}_{K}=\sum_{j,k}a_j b_k
K(x_j,y_k)$. We denote by $|\cdot|_K$ the associated norm on $\cH_K$.

The representer theorem (see~\cite{kimeldorf_wahba71} for results on
optimization in RKHS) is at the heart of minimization of functional
onto RKHS. The solution of the minimization problem
\begin{equation}
  \label{eq:RKHS_estimator}
  \bar{f} \in \argmin_{f \in \cH_K} \{ R_n(f) + h^2|f|_{\cH_K}^2 \}
\end{equation}
is the linear combination
\begin{equation*}
  \bar{f} (\cdot) = \sum_{i=1}^n \alpha_i K(X_i,\cdot),\mbox{ where }
  \boldsymbol {\alpha} = (\alpha_i)_{1 \leq i \leq n} = (\mathbf K_X +
  n h^2 \mathbf I_n)^{-1} \mathbf Y,
\end{equation*}
where $\mathbf K_X$ is the Gram matrix $(K(X_i,X_j))_{1\leq i,j\leq
  n}$, where $\mathbf Y = (Y_1, \ldots, Y_n)$ and where $\mathbf I_n$
is the identity matrix in $\mathbb R^n$. They are many different ways
to simplify the computation of the coefficients $\boldsymbol{\alpha}$,
see for instance~\cite{amato_antoniadis_pensky06}.

In order to derive convergence rates for the estimator defined
in~\eqref{eq:RKHS_estimator} from Theorem~\ref{thm:least_sq}, we use
some results about covering numbers of RKHS obtained
in~\cite{cucker_smale02} (other results on the entropy of RKHS can be
found in \cite{SS:07,CS:98}). Let now assume that $P_X$ is a Borel
measure. If $K$ is a {\it Mercer kernel} (this is a continuous
reproducing kernel), the RKHS associated with $K$ is the set
\begin{equation*}
  \label{eq:Mercer_kernel}
  \cH_K=\Big\{f\in L_2(P_X): f=\sum_{j=1}^\infty a_j \psi_j \mbox{
    s.t. } \sum_{j=1}^\infty \lambda_j^{-1} a_j^2\leq \infty\Big\},
\end{equation*}
where $(\lambda_j)_{j\geq1}$ is the sequence of decreasing eigenvalues
of the operator
\begin{equation*}
  L_K:\left\{\begin{array}{ccc}
      L^2(P_X) & \longrightarrow & L^2(P_X)\\
      f        & \longmapsto     & \int_\cX K(\cdot,y)f(y)dP_X(y)
    \end{array} \right.
\end{equation*}
and $(\psi_j)_{j\leq1}$ the sequence of corresponding
eigenvectors. According to Proposition~9 and Theorem~D in
\cite{cucker_smale02}, if for any $k\geq1$ the $k$-th eigenvalue of
$L_K$ is such that
\begin{equation}
  \label{eq:rkhs_eigenvalue}
  \lambda_k \leq C k^{-l}
\end{equation}
for some $C > 0$ and $l > 1/2$ then the entropy of $B_K(R) := \{f \in
\cH_K : |f|_K \leq R\}$ satisfies for any $\delta > 0$:
\begin{equation*}
  H_\infty(\delta, B_K(R)) \leq \Big(\frac{2 R C_l}{\delta}
  \Big)^{1/l},
\end{equation*}
where $C_l$ is slightly greater than $6Cl^l$. In this case,
Theorem~\ref{thm:least_sq} and the arguments from
Section~\ref{sec:derive_adaptive} gives the following result.

\begin{corollary}[Adaptive upper bound for RKHS]
  \label{cor:rkhs}
  Let $\bar f$ be defined by~\eqref{eq:RKHS_estimator} with a
  reproducing kernel $K$ such that the eigenvalues of the operator
  $L_K$ satisfy~\eqref{eq:rkhs_eigenvalue}. Then, if $h = a n^{-l /
    (2l + 1)}$ and $\norm{\bar f - f_0}_\infty \leq Q$, we have
  \begin{equation*}
    E^n \norm{\bar f - f_0}_{L^2(P_X)}^2 \leq C_2 (1 +
    |f_0|^2_{\mathcal H_K}) n^{-2l / (2l + 1)}
  \end{equation*}
  when $n$ is large enough.

  Now, let $L = [l_{\min}, l_{\max}]$ where $l_{\min} > 1/2$ and
  $(\mathcal H_l : l \in L)$ be a family of nested RKHS. Assume that
  the kernel of each $\mathcal H_l$
  satisfies~\eqref{eq:rkhs_eigenvalue}. Let $\hat {\mathsf f}_n$ be
  the aggregated estimator defined by Steps~1-3 with $\Lambda_n = \{
  \lambda = (n^{-l / (2l + 1)}, \mathcal H_l) : l \in L_n \}$ and $L_n
  := \{ l_{\min}, l_{\min} + (\log n)^{-1}, \ldots, l_{\max} \}$. We
  have, if $f_0 \in \mathcal H_l$ for some $l \in L$,
  \begin{equation*}
    E^n \norm{\hat {\mathsf f}_n - f_0}_{L^2(P_X)}^2 \leq C_2 (1 +
    |f_0|^2_{\mathcal H_l}) n^{-2l / (2l + 1)}
  \end{equation*}
  when $n$ is large enough.
\end{corollary}

\subsection{Anisotropic Besov spaces}
\label{sec:anisotropic_besov}

In nonparametric estimation literature, Besov spaces are of particular
interest since they include functions with \emph{inhomogeneous
  smoothness}, for instance functions with rapid oscillations or
bumps. Roughly, these spaces are used in statistics when we want to
prove theoretically that some adaptive estimator is able to recover
the details of a functions. When one considers a multivariate
regression, the question of anisotropic smoothness naturally arises.
Anisotropy means that the smoothness of $f_0$ differs in function of
coordinates. As far as we know, adaptive estimation of a multivariate
curve with anisotropic smoothness was previously considered only in
Gaussian white noise or density models, see~\cite{hoffmann_lepski02},
\cite{kerk_lepski_picard01}, \cite{kerk_lepski_picard07},
\cite{neumann00}.  There is no results concerning the adaptive
estimation of the regression with anisotropic smoothness on a general
random design.

In this Section, we construct, using Steps~1-3, an adaptive estimator
over anisotropic Besov spaces $B_{p, q}^{\bs s}$, where $\bs s = (s_1,
\ldots, s_d)$ is the vector of smoothnesses. If $\{ e_1, \ldots, e_d
\}$ is the canonical basis of $\mathbb R^d$, each $s_i$ is the
smoothness in the direction $e_i$. A precise definition of $B_{p,
  q}^{\bs s}$ is given in
Appendix~\ref{sec:appendix_approximation}. Let $s$ be the harmonic
mean of $\bs s$, see~\eqref{eq:harmonic_mean}. Let us introduce two
vectors $\bs s^{\min}$ and $\bs s^{\max}$ in $\mathbb R_+^d$ with
positive coordinates and harmonic means $\bar {\bs s}^{\min}$ and
$\bar {\bs s}^{\max}$ respectively. Assume that $\bs s^{\min} \leq
{\bs s}^{\max}$, which means that $s_i^{\min} \leq s_i^{\max}$ for any
$i \in \{ 1, \ldots, d \}$ and assume that $\bar {\bs s}^{\min} > d /
\min(p, 2)$. In view of Theorem~\ref{thm:anisotropic_entropy} and the
embedding~\eqref{eq:anisotropic_embedding} (see
Appendix~\ref{sec:appendix_approximation}), we know that Assumption
$(C_\beta)$ holds for every $B_{p, \infty}^{\bs s}$ such that $\bs s
\geq \bs s^{\min}$ with $\beta = d / \bar {\bs s}$ (and every $B_{p,
  q}^{\bs s}$, since $B_{p, q}^{\bs s} \subset B_{p, \infty}^{\bs
  s}$), where $\bar {\bs s}$ is the harmonic mean of $\bs s$. Consider
the ``cube of smoothness''
\begin{equation}
  \label{eq:smoothness_cube}
  \bs S := \prod_{i=1}^d [s_i^{\min}, s_i^{\max}],
\end{equation}
and consider the uniform discretization of this cube with step $(\log
n)^{-1}$:
\begin{equation}
  \label{eq:discr_smoothness_cube}
  \bs S_n := \prod_{i=1}^d  \big\{ s_i^{\min}
  + k (\log n)^{-1} :1\leq k \leq [ (s_i^{\max} - s_i^{\min}) \log n ]
  \big\},
\end{equation}
and the set of parameters
\begin{equation*}
  \Lambda(\bs S) := \{ \lambda = (n^{- \bar {\bs s} / (2 \bar {\bs s}
    + d)}, B_{p, q}^{\bs s}) : \bs s \in \bs S_n \}.
\end{equation*}
Now, we compute, following Steps~1-3, the aggregated estimator $\hat
{\mathsf f}_n^{\bs S}$ with set of parameters $\Lambda(\bs S)$ (see
the beginning of the section). Following the arguments from
Section~\ref{sec:derive_adaptive}, we can prove in the following
Corollary~\ref{cor:anisotropic_besov_rate} that $\hat {\mathsf
  f}_n^{\bs S}$ is adaptive over the whole range of anisotropic Besov
spaces $\{ B_{p, q}^{\bs s} : \bs s \in \bs S \}$.





\begin{corollary}
  \label{cor:anisotropic_besov_rate}
  Assume that $\norm{\bar f - f_0}_\infty \leq Q$ for every $\bar f
  \in F(\bs S)$. If $f_0 \in B_{p, q}^{\bs s}$ for some $s \in \bs S$,
  then
  \begin{equation*}
    E^n \norm{\hat {\mathsf f}_n^{\bs S} - f_0}_{L^2(P_X)}^2 \leq C
    n^{-2 \bar {\bs s} / (2 \bar {\bs s} + d)}
  \end{equation*}
  when $n$ is large enough, where $C$ is a constant depending on $\bs
  S, d$ and $Q$.
\end{corollary}

In Corollary~\ref{cor:anisotropic_besov_rate} we recover the
``expected'' minimax rate $n^{-2 \bar {\bs s} / (2 \bar {\bs s} + d)}$
of estimation of a $d$-dimensional curve in a Besov space. Note that
there is no regular or sparse zone here, since the error of estimation
is measured with $L^2(P_X)$ norm. A minimax lower bound over $B_{p,
  q}^{\bs s}$ can be easily obtained using standard arguments, such as
the ones from~\cite{tsybakov03}, together with Bernstein estimates
over $B_{p, q}^{\bs s}$ that can be found in~\cite{hochmuth02}. Note
that the only assumption required on the design law in this corollary
is the compactness of its support.

\section{Empirical study}
\label{sec:simulations}

In this Section, we compare empirically our aggregation procedure with
the popular cross-validation (CV) and generalized cross-validation
(GCV) procedures for the selection of the smoothing parameter $h$ (see
Section~\ref{sec:about_h}) in smoothing splines (we use the
\texttt{smooth.spline} routine from the \texttt{R} software, see
\texttt{http://www.r-project.org/}). Concerning CV, GCV and smoothing
splines, we refer to~\cite{wahba90}
and~\cite{green_silverman94}. Those routines provide satisfactory
results in most cases, in particular for the examples of regression
functions considered here. However, we show that when the sample size
$n$ is small (less than 50), and when the noise level is high (we take
root-signal-to-noise ratio equals to $2$), then our aggregation
approach is more stable, see Figure~\ref{fig:mises} below. Here in, we
consider two examples of regression function, given, for $x \in [-1,
1]$, by:
\begin{itemize}
\item \texttt{hardsine}$(x) = 2 \sin(1 + x) \sin( 2 \pi x^2 + 1)$
\item \texttt{oscsine}$(x) = (x+1) \sin(4 \pi x^2 )$.
\end{itemize}
We simply take $X$ uniformly distributed on $[-1, 1]$ and Gaussian
noise with variance $\sigma$ chosen so that the root-signal-to-noise
ratio is $2$. In Figure~\ref{fig:examples} we show typical simulation
in this setting, where $n = 30$.
\begin{figure}[htbp]
  \centering
  \includegraphics[width=6cm]{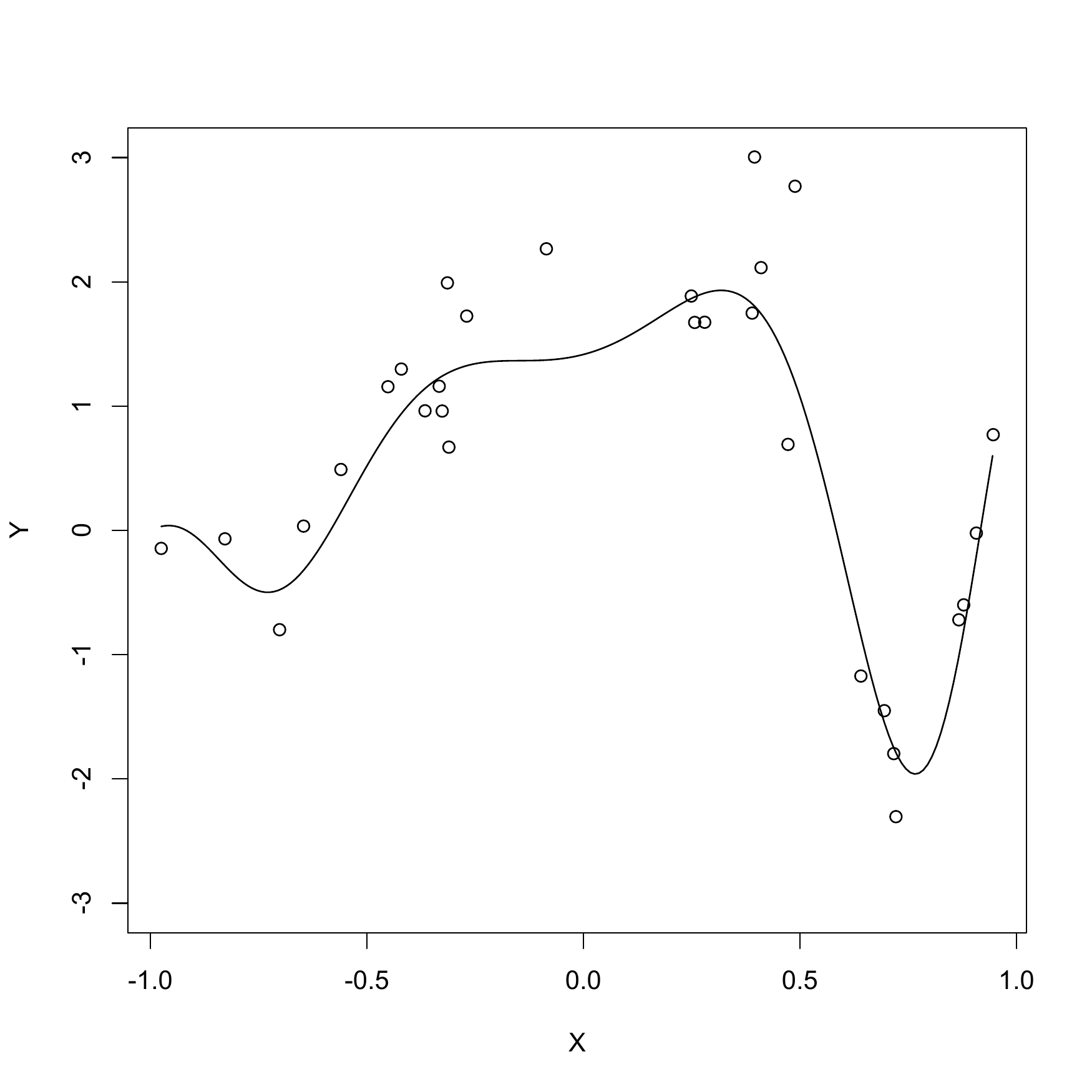}%
  \includegraphics[width=6cm]{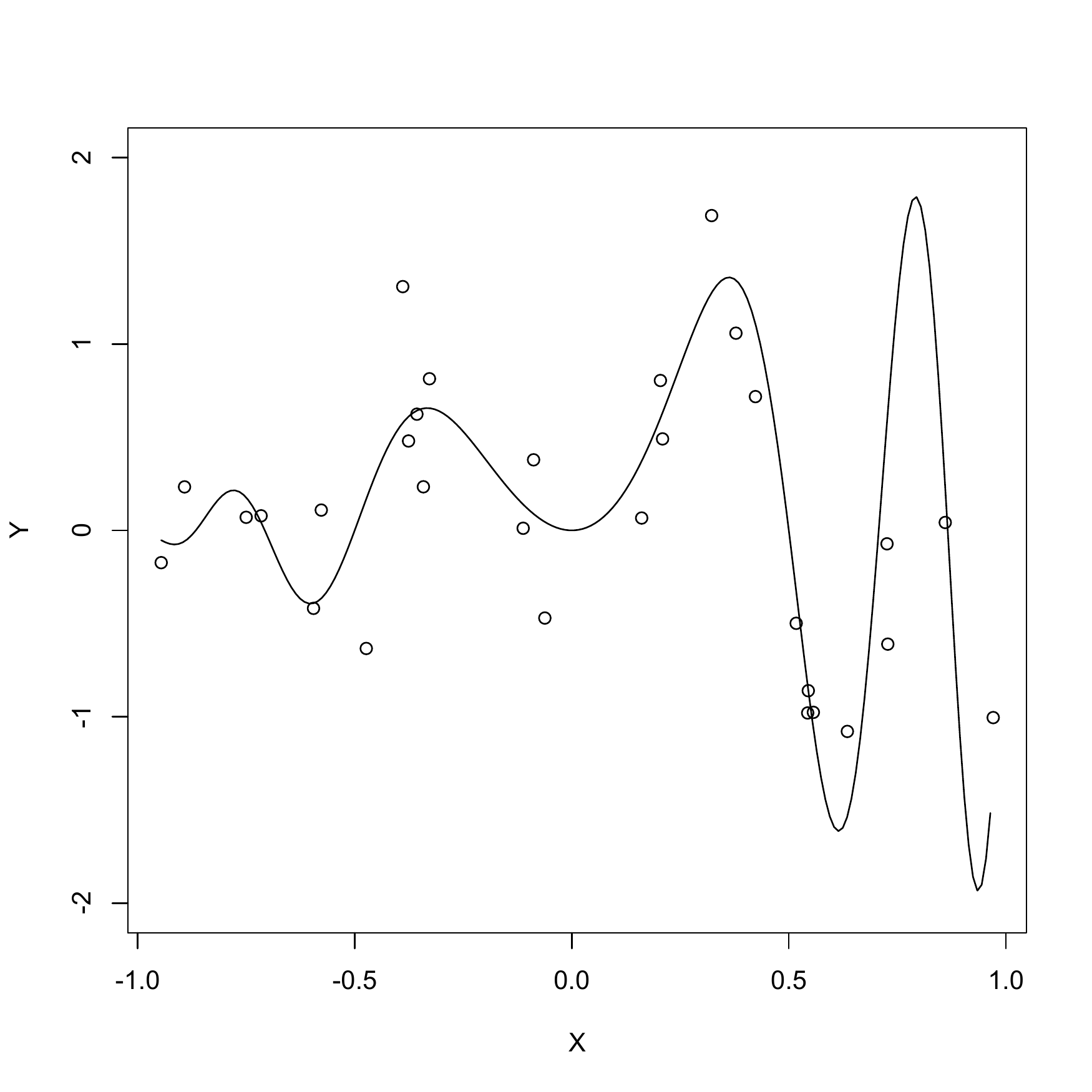}%
  \caption{Examples of simulated data, for
    $f_0$\texttt{=\textup{harsine}} \textup(left\textup) and
    $f_0$\texttt{=\textup{oscsine}} \textup(right\textup)}
  \label{fig:examples}
\end{figure}

In Figure~\ref{fig:mises}, we show the mises $E\norm{\hat f_n -
  f_0}_n^2$ computed by Monte Carlo using $1000$ simulations of the
model. The tuning of the estimators in both examples is the following:
for GCV, we simply use the \texttt{smooth.spline} routine with default
selection of $h$ by GCV. For CV, we use the same routine, with the
option \texttt{cv=TRUE} so that CV is used instead. For aggregation,
we use Steps~1-3 (see Section~\ref{sec:examples}). Step~1 is done with
$m=3n/4$ and $\ell = n/4$. For Step~2, we use the
\texttt{smooth.spline} routine to compute a set of weak estimators,
using the option \texttt{spar=x}, where \texttt{x} lies in the set $\{
0, 0.01, 0.02 \ldots, 1 \}$. The parameter \texttt{spar} is related to
the value of the smoothing parameter $h$. For Step~3, we compute the
weights with temperature given by~\eqref{Tslection} (over the training
sample) and the set $\mathcal T = \{ 10, 20, \ldots, 100 \}$. Then, we
repeat steps~1-3 $J=100$ times and compute the jackknifed estimator,
see Section~\ref{sec:jackknife}. This gives our aggregated estimator.

On Figure~\ref{fig:mises}, we plot the MISEs (the mean of the $1000$
MISEs obtained for each simulation) for sample sizes $n \in \{ 20, 30,
50, 100 \}$ and in Figure~\ref{fig:sd} we plot the corresponding
standard deviations. The conclusion is that for small $n$, aggregation
provides a more accurate and stable estimation than the GCV or
CV. When $n$ is $100$ or larger, than the aggregation procedure has
barely the same accuracy as GCV or CV.

\begin{figure}[htbp]
  \centering
  \includegraphics[width=6cm]{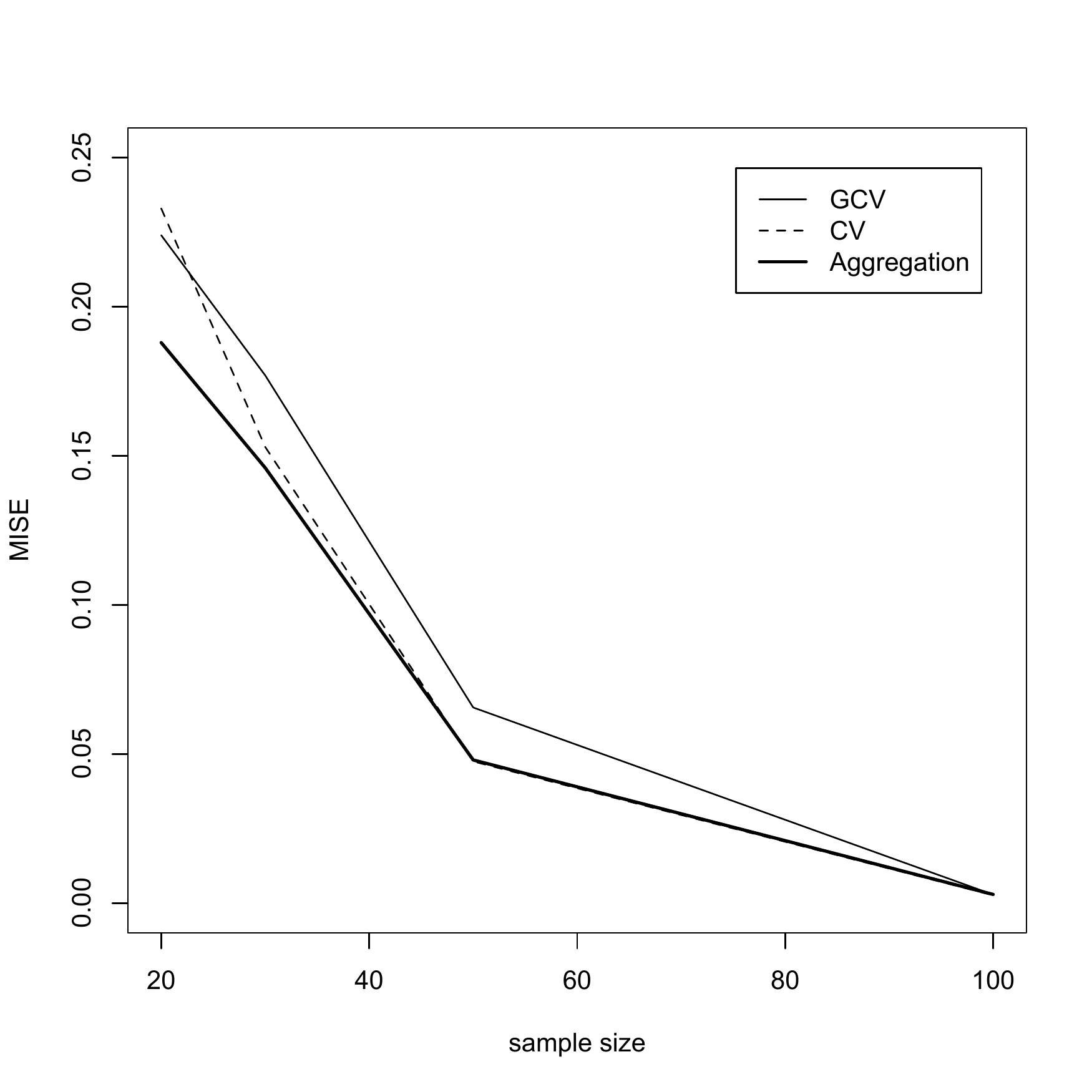}%
  \includegraphics[width=6cm]{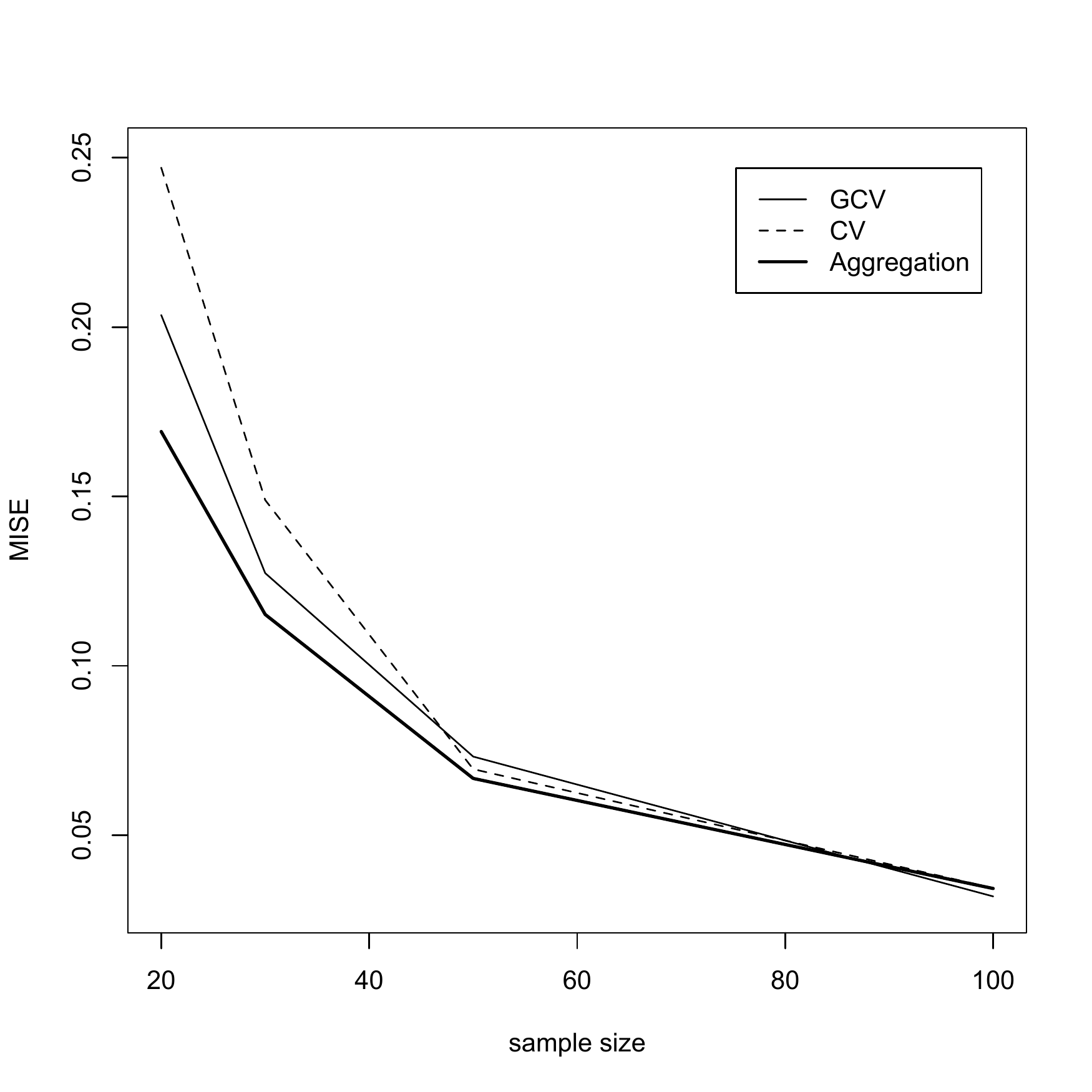}%
  \caption{MISE for $f_0$\textup{=\texttt{harsine}}
    \textup(left\textup) and $f_0$\textup{=\texttt{oscsine}}
    \textup(right\textup)}
  \label{fig:mises}
\end{figure}

\begin{figure}[htbp]
  \centering
  \includegraphics[width=6cm]{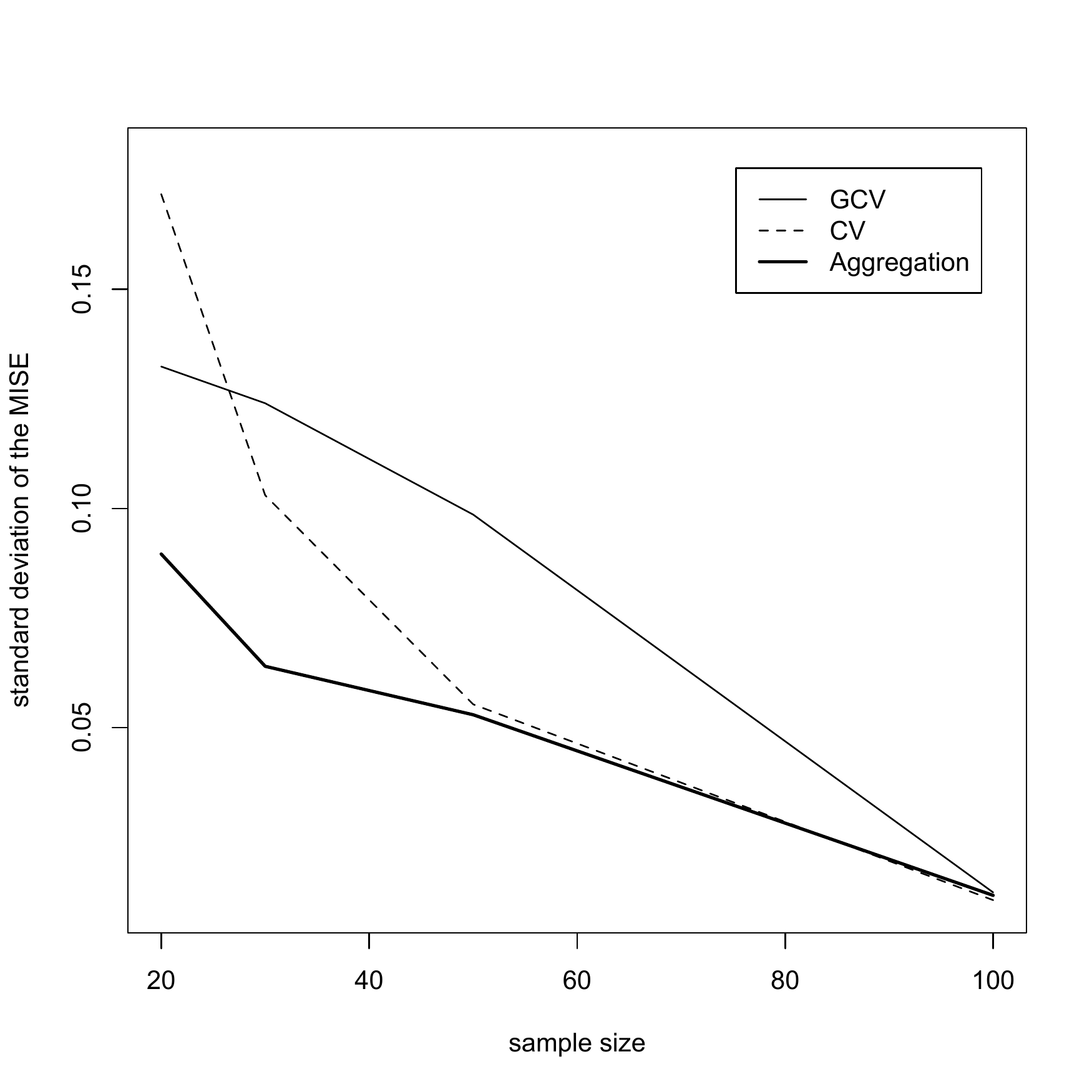}%
  \includegraphics[width=6cm]{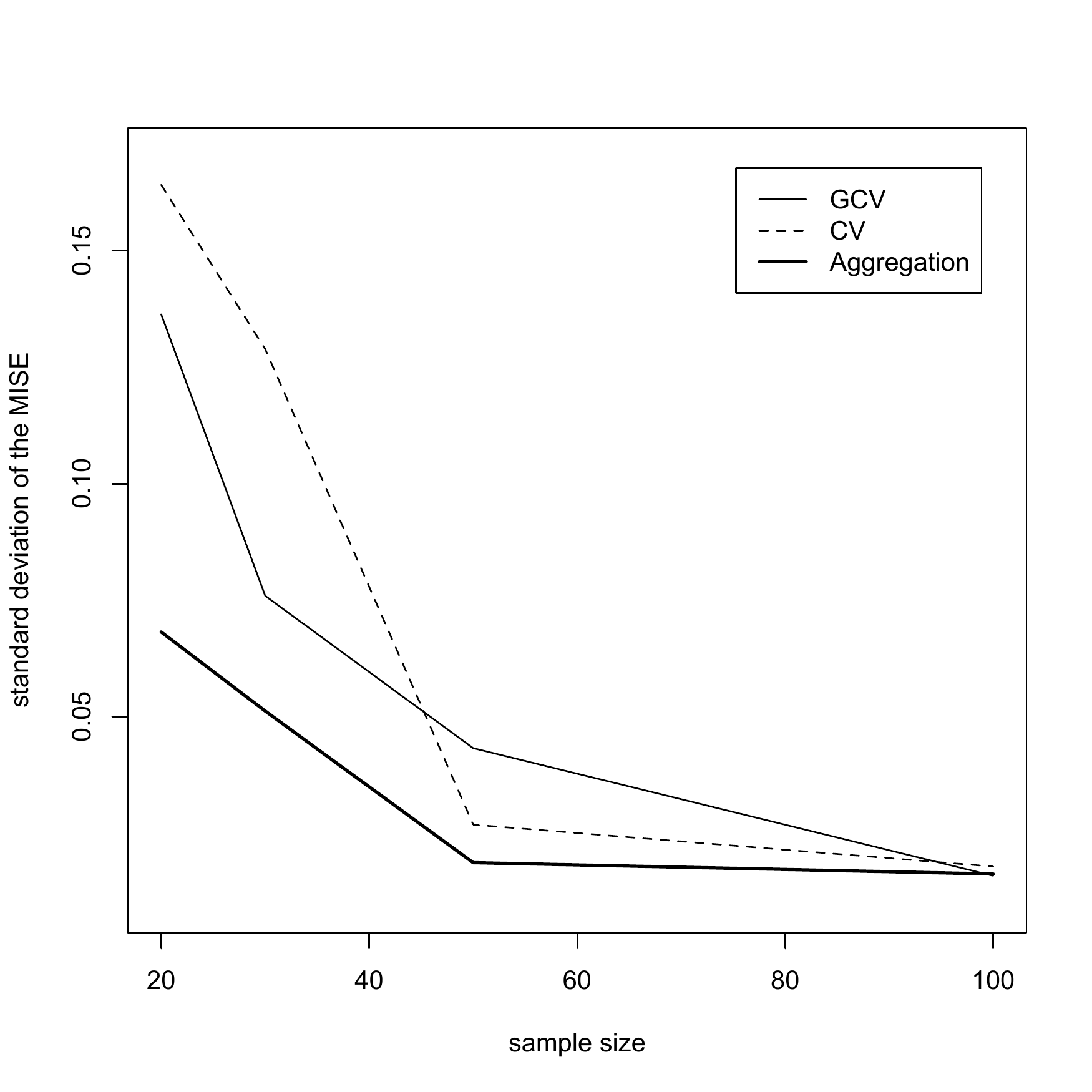}%
  \caption{standard deviation of the MISE for
    $f_0$\textup{=\texttt{harsine}} \textup(left\textup) and
    $f_0$\textup{=\texttt{oscsine}} \textup(right\textup)}
  \label{fig:sd}
\end{figure}

\section{Proofs of the main results}
\label{sec:proof_main_results}

We recall that $P_n$ stands for the joint law of the training sample
$D_n$ conditional on $X^n := (X_1, \ldots, X_n)$, that is $P_n :=
P^n[\cdot | X^n]$.


\begin{proof}[Proof of Theorem~\ref{thm:devia1}]
  First, we use the \emph{peeling} argument: we decompose $B_n(f_0,
  \delta)$ into the union of the sets $S_j$ for $j \geq 0$, where for
  $\delta_j := \delta 2^{-j/\beta}$
  \begin{equation*}
    S_j := B_n(f_0, \delta_j ) - B_n(f_0, \delta_{j+1}),
  \end{equation*}
  and decompose $\mathcal F$ into the union of the sets
  \begin{equation*}
    B_\cF(2^{k/\beta}) - B_\cF(2^{(k-1)/\beta}) = \{ f \in \mathcal F
    : 2^{(k-1) / \beta} < |f|_{\mathcal F} \leq 2^{k / \beta} \},
 \end{equation*}
 for $k \geq 1$, where $B_{\mathcal F}(2^{k/\beta}) = \{ f \in
 \mathcal F : |f|_{\mathcal F} \leq 2^{k/\beta}\}$ This gives that the
 left hand side of~\eqref{eq:deviaZ_n} is smaller than
   \begin{align*}
     \sum_{j \geq 0} & P_n\Big[ \sup_{ \substack{f \in S_j \text{
           s.t. } \\ |f|_{\mathcal F} \leq 1} } \frac{ Z(f - f_0)
     }{\norm{f - f_0}_n^{1 - \beta/2} (1 + |f|_{\mathcal F})^{\beta/2}
     } > z \Big] \\
     &+ \sum_{j \geq 0} \sum_{k \geq 1} P_n \Big[ \sup_{ f
         \in S_j\cap B_{\mathcal F}(2^{k/\beta})} \frac{ Z(f - f_0) }{\norm{f -
         f_0}_n^{1 - \beta/2} (1 + |f|_{\mathcal F})^{\beta/2} } > z
     \Big],
   \end{align*}
   which is smaller than
   \begin{equation*}
     \sum_{j,k \geq 0}P_n \Big[ \sup_{f \in
         B_n(f_0, \delta_j)\cap B_\cF(2^{k/\beta}) }
     Z(f - f_0) > z(\delta, j, k) \Big] =: \sum_{j,k \geq 0}  P_{j, k},
   \end{equation*}
   where $z(\delta, j, k) := z \delta_j^{1 - \beta/2}
   2^{k/2-1/2}$. Let us consider, for any $\delta > 0$, a minimal
   $\delta$-covering $F(\delta, k)$ of the set $B_{\mathcal
     F}(2^{k/\beta})$ for the
   $\norm{\cdot}_\infty$-norm. Assumption~$(C_\beta)$ implies
   \begin{equation*}
     | F(\delta, k) | \leq \exp\big( D (2^{k/\beta} / \delta)^{\beta} \big)
     = \exp( D 2^k \delta^{-\beta} ).
   \end{equation*}
   Moreover, without loss of generality, we can assume that $F(\delta,
   k) \subset B_{\mathcal F}(2^{k/\beta})$. For any $i \in \mathbb N$
   and $j, k$ fixed, we introduce
   \begin{equation}
     \label{eq:Fi}
     F^{(i)} := F(\delta_{i,j}, k) \text{ where } \delta_{i,j} :=
     \delta_j 2^{-i/\beta} = \delta 2^{-(i+j)/\beta},
   \end{equation}
   and, for any $f\in B_\cF(2^{k/\beta})$ we denote by $\pi_i(f)$ an
   element of $F^{(i)}$ such that $\norm{\pi_i(f) - f}_\infty \leq
   \delta_{i,j}$. We have
   \begin{align*}
     P_{j,k} &\leq P_n\Big[ \sup_{ f \in B_n(f_0, \delta_j)\cap
       B_\cF(2^{k/\beta})} | Z(\pi_0(f) - f_0) | > z(\delta, j, k) / 2
     \Big] \\ & + P_n \Big[ \sup_{ f \in B_n(f_0, \delta_j) \cap
       B_\cF(2^{k/\beta})} | Z(f - \pi_0(f))| > z(\delta, j, k) / 2
     \Big] \\ &=: P_{j,k,1} + P_{j,k,2}.
   \end{align*}
   First, we consider $P_{j,k,1}$:
   \begin{align*}
     P_{j,k,1} \leq P_n \Big[ \sup_{f \in F^{(0)} \cap B_n(f_0,
       \delta_j) } | Z(\pi_0(f) - f_0) | > z(\delta, j, k) / 2 \Big].
   \end{align*}
   We use~\eqref{eq:deviaZnf} and the union bound over $F^{(0)}$
   together with the fact that $f \in B_n(f_0, \delta_j)$ to obtain:
   \begin{equation*}
     P_{j,k,1} \leq |F^{(0)}| \exp\Big( \frac{-a z^2(\delta, j, k)}{4
       \delta_j^2} \Big) = \exp\Big( \frac{2^{j+k}}{\delta^{\beta}} (D - a z^2 / 8 ) \Big),
   \end{equation*}
   where $a := (2b^2)^{-1}$. Now, in order to control $P_{j,k,2}$, we
   use the so-called chaining argument, which involves increasing
   approximations by the covers $F^{(i)}$, see~\eqref{eq:Fi}. Let us
   consider
   \begin{equation*}
     E_i := (2^{1/\beta - 1/2} - 1) 2^{-i (1/\beta-1/2) }
   \end{equation*}
   for $i \geq 1$ ($E_i > 0$ since $\beta \in(0, 2)$). By linearity of
   $Z_n(\cdot)$ and since $\sum_{i \geq 1} E_i = 1$, we have
   \begin{align*}
     P_{j,k,2} &\leq \sum_{i \geq 1} P_n\Big[ \sup_{ \substack{ f \in
         B_n(f_0, \delta_j) \\ |f|_{\mathcal F} \leq 2^{k/\beta} } } |
     Z(\pi_i(f) - \pi_{i-1}(f)) | > E_i z(\delta, j, k) / 2 \Big] \\
     &=: \sum_{i \geq 1} P_{i, j, k, 2}.
   \end{align*}
   Now, since
   \begin{align*}
     \norm{\pi_i(f) - \pi_{i-1}(f)}_n &\leq \norm{\pi_i(f) -
       \pi_{i-1}(f)}_\infty \\
     &\leq \norm{\pi_i(f) - f}_\infty + \norm{\pi_{i-1}(f) - f}_\infty \\
     & \leq \delta_{i,j} + \delta_{i-1,j} = \delta_{i,j} (1 +
     2^{1/\beta}),
   \end{align*}
   and since the number of pairs $\{ \pi_i(f), \pi_{i-1}(f) \}$ is at
   most
   \begin{equation*}
     |F^{(i)}| \times |F^{(i-1)}| \leq \exp \Big( \frac{3 D 2^{i + j +
         k}}{2 \delta^{\beta}} \Big),
   \end{equation*}
   we obtain using again~\eqref{eq:deviaZnf}:
   \begin{align*}
     P_{i, j, k, 2} &\leq |F^{(i)}| \times |F^{(i-1)}| \times
     \exp\Big( \frac{-a E_i^2 z^2(\delta, j, k)}{4 \delta_{i,j}^2 (1 +
       2^{1/\beta})^2} \Big) \\
     &= \exp\Big( \frac{2^{i+j+k}}{\delta^{\beta}} \big( 3 D / 2 - C_1
     z^2 \big) \Big)
   \end{align*}
   where $C_1 = C_1(s, d, a) := a(2^{1/\beta -
     1/2} - 1) / (8 (1 + 2^{1/\beta})^2) > 0$. Then, if we choose $z_1
   := (3 / C_1)^{1/2}$, we have for any $z \geq z_1$ and $D_1 := C_1 /
   2$:
   \begin{align*}
     \sum_{j, k \geq 0} P_{j,k} &\leq \sum_{j,k \geq 0} \Big(
     P_{j,k,1} + \sum_{i \geq 1} P_{i,j,k,2} \Big) \\
     &\leq \sum_{j,k \geq 0} \Big( \exp( -D_1 2^{j+k} z^2
     \delta^{-\beta} ) + \sum_{i \geq 1} \exp( -D_1 2^{i+j+k} z^2
     \delta^{-\beta} ) \Big)
   \end{align*}
   and the Theorem follows.
\end{proof}

\begin{proof}[Proof of Theorem~\ref{thm:least_sq}]
  For short, we shall write $\bar f$ instead of $\bar f_\lambda$, and
  $\pen(f)$ instead of $\pen_\lambda(f)$. In view
  of~\eqref{eq:pena_least_sq}, we have
  \begin{equation}
    \label{eq:f_bar_prop}
    \norm{Y - \bar f}_n^2 + \pen(\bar f) \leq \norm{Y - f}_n^2 +
    \pen(f) \quad \forall f \in \mathcal F,
  \end{equation}
  which is equivalent to
  \begin{equation*}
    \norm{\bar f - f}_n^2 + \pen(\bar f) \leq 2 \prodsca{Y -
      f}{\bar f - f}_n + \pen(f) \quad \forall f \in \mathcal F,
  \end{equation*}
  where $\prodsca{f}{g}_n = n^{-1} \sum_{i=1}^n f(X_i) g(X_i)$. This
  entails, since $f_0 \in \mathcal F$, that
  \begin{equation}
    \label{eq:trick1}
    \norm{\bar f - f_0}_n^2 + \pen(\bar f) \leq
    \frac{2}{\sqrt{n}} Z(\bar f - f_0) + \pen(f_0)
  \end{equation}
  where $Z(\cdot)$ is the empirical process given
  by~\eqref{eq:Z_n_def}. Recall that $B_n(f_0, \delta)$ stands for the
  ball centered at $f_0$ with radius $\delta$ for the norm
  $\norm{\cdot}_n$. Let us introduce the event
  \begin{equation}
    \label{eq:event_Z}
    \mathcal Z(z, \delta) := \Big\{ \sup_{f \in \mathcal F \cap
      B_n(f_0, \delta)} \frac{ Z(f - f_0) }{\norm{f - f_0}_n^{1 -
        \beta/2} (1 + |f|_{\mathcal F})^{\beta/2} } \leq z \Big\}.
  \end{equation}
  In view of Theorem~\ref{thm:devia1}, see
  Section~\ref{sec:process_Z0}, we can find constants $z_1 > 0$ and
  $D_1 > 0$ such that\textup:
  \begin{align*}
    P_n\big[ \mathcal Z(z, \delta)^\complement \big] \leq \exp( - D_1
    z^2 \delta^{-\beta} ),
  \end{align*}
  for any $\delta > 0$ and $z \geq z_1$. When $2 n^{-1/2} Z(\bar f -
  f_0) \leq \pen(f_0)$, we have $\norm{\bar f - f_0}_n^2 \leq 2
  \pen(f_0)$. When $2 n^{-1/2} Z(\bar f - f_0) \geq \pen(f_0)$, we
  have, for any $z>0$, in view of~\eqref{eq:trick1}, whenever $\bar f \in B_n(f_0,
  \delta)$ for some $\delta > 0$, that on $\mathcal Z(z, \delta)$,
  \begin{equation*}
    \norm{\bar f - f_0}_n^2 + \pen(\bar f) \leq \frac{4 z}{\sqrt{n}}
    \norm{\bar f - f_0}_n^{1 - \beta/2} (1 + |\bar f|_{\mathcal
      F})^{\beta/2}.
  \end{equation*}
  If $|\bar f|_{\mathcal F} \leq 1$, this entails
  \begin{equation*}
    \norm{\bar f - f_0}_n^2 + \pen(\bar f) \leq ( a^{-2}(2^\beta 4
    z)^{4 / (2 + \beta)} + 1) h^2.
  \end{equation*}
  Otherwise, we have
  \begin{equation*}
    \norm{\bar f - f_0}_n^2 + \pen(\bar f) \leq \frac{2^{\beta/2} 4
      z}{\sqrt{n}} \norm{\bar f - f_0}_n^{1 - \beta/2} |\bar
    f|_{\mathcal F}^{\beta/2},
  \end{equation*}
  and we use the following lemma.
  \begin{lemma}
    \label{lem:logtrick}
    Let $r, I, h, \varepsilon$ be positive numbers, $\beta \in (0, 2)$
    and $\alpha > 2\beta / (\beta + 2)$. Then, if
    \begin{equation}
      \label{eq:logtrick}
      r^2 + h^2 I^\alpha \leq \varepsilon\, r^{1 - \beta / 2} I^{\beta / 2},
    \end{equation}
    we have
    \begin{equation*}
      r \leq ( \varepsilon^{\alpha} h^{-\beta} )^{2 / (2\alpha + \alpha \beta
        -  2 \beta)}, \quad I \leq (\varepsilon^2
      h^{-(\beta + 2)})^{2 / (2 \alpha + \alpha \beta - 2\beta)}
    \end{equation*}
    and consequently
    \begin{equation*}
      r^2 + h^2 I^\alpha \leq 2 (\varepsilon^\alpha
      h^{-\beta})^{4/(2\alpha + \alpha \beta - 2\beta)}.
    \end{equation*}
  \end{lemma}
  The proof of this Lemma is given in Section~\ref{sec:lemmas_proofs}
  below. It entails, since $h = a n^{-1 / (2 + \beta)}$ and $\alpha >
  2\beta / (\beta+2)$, that
  \begin{equation*}
    \norm{\bar f - f_0}_n^2 + h^2 |\bar f|_{\mathcal F}^{\alpha} \leq
    2 ((2^{\beta/2} 4 z)^{\alpha} a^{-\beta})^{4 / (2\alpha + \alpha
      \beta - 2\beta)} n^{-2 / (\beta+2)}.
  \end{equation*}
   Thus,
  when $\bar f \in B_n(f_0, \delta)$, we have on $\mathcal Z(z, \delta)$:
  \begin{equation*}
    \norm{\bar f - f_0}_n^2 + \pen(\bar f) \leq p(z)^2 h^2
  \end{equation*}
  where
  \begin{equation*}
    p(z)^2 := C_1 (1 + z^{4 / (2 + \beta)} + z^{4\alpha / (2\alpha + \alpha
      \beta - 2\beta)})
  \end{equation*}
  and $C_1$ is a constant depending on $\alpha, \beta$ and $a$.
  Let us assume for now that $\norm{\bar f - f_0}_n \leq \delta$ for
  some $\delta > 0$, and let us introduce
  \begin{equation*}
    \mathcal Z_1(z, \delta) := \mathcal Z(z, \delta) \cap \mathcal
    Z(z_1, p(z) h),
  \end{equation*}
  where $z_1$ is a constant coming from Theorem~\ref{thm:devia1}. On
  $\mathcal Z_1(z, \delta)$, we have
  \begin{equation}
    \label{eq:on_Z1}
    \norm{\bar f - f_0}_n^2 + \pen(\bar f) \leq p(z_1)^2 h^2.
  \end{equation}
  Indeed, we have $\bar{f}\in B_n(f_0,\delta)$ thus, on $\mathcal Z(z, \delta)$, $\norm{\bar f - f_0}_n^2 + \pen(\bar f) \leq p(z_1)^2 h^2$ and so
  $\norm{\bar f - f_0}_n^2\leq p(z)^2 h^2$. Thus, on the event $\mathcal
    Z(z_1, p(z) h)$, we have (\ref{eq:on_Z1}). Moreover, Theorem~\ref{thm:devia1} yields
  \begin{equation}
    \label{eq:deviaB1}
    P_n \big[ \mathcal Z_1( z, \delta)^\complement \big] \leq \exp(
    -D_1 z^2 \delta^{-\beta}) + \exp( -D_1 z_1^2 (p(z) h)^{-\beta}).
  \end{equation}
  Now, in view of~\eqref{eq:f_bar_prop} and since $f_0 \in \mathcal
  F$, we have the following rough majoration:
  \begin{align}
    \nonumber \norm{\bar f - f_0}_n^2 + \pen(\bar f) &\leq 2
    (\norm{\bar f - Y}^2_n + \pen(\bar f) ) + 2 \norm{f_0 - Y}_n^2
    \\ \nonumber &\leq 2 ( \norm{f_0 - Y}_n^2 + \pen(f_0)) + 2
    \norm{f_0 - Y}_n^2 \\
    \label{eq:rough1}
    &\leq 4 \sigma^2 \norm{\varepsilon}_n^2 + 2 \pen(f_0),
  \end{align}
  which entails
  \begin{equation*}
    E_n\big[ \big( \norm{\bar f - f_0}_n^2 + \pen(\bar f) \big)^2
    \big] \leq \sigma^4 C(\varepsilon)^2 + 8 h^4 |f_0|_{\mathcal F}^{2\alpha}
  \end{equation*}
  where $C(\varepsilon)^2 = 32( E[\varepsilon^4] / n + 2
  (E[\varepsilon^2])^2)$. Putting all this together, we obtain, by a
  decomposition of $E_n[\norm{\bar f - f_0}_n^2 + \pen(\bar f)]$ over
  the union of the sets $\{ \norm{\bar f - f_0}_n\leq \delta \} \cap
  \mathcal Z_1(z, \delta)$, $\mathcal Z_1(z, \delta)^\complement$ and
  $\{\norm{\bar f - f_0}_n > \delta \}$ that
  \begin{align*}
    E_n[ \norm{\bar f - &f_0}_n^2 + \pen(\bar f)] \leq p(z_1)^2 h^2 \\
    &+ (\sigma^2 C(\varepsilon) + 2\sqrt{2} h^2
    |f_0|_{\mathcal F}^\alpha)\big(
    P_n[ \mathcal Z_1(z, \delta)^\complement]^{1/2}+P_n[ \norm{\bar f - f_0}_n > \delta]^{1/2}\big).
  \end{align*}
  In view of~\eqref{eq:rough1}, if $\delta > 2 \pen(f_0)\vee1$ then we have
  $\{ \norm{\bar f - f_0}_n^2 > \delta^2 \} \subset \{
  \norm{\varepsilon}_n^2 > (\delta^2 - \delta) / (4 \sigma^2)\}$.
  Thus, using the subgaussianity assumption~\eqref{eq:subgaussian}, we
  have $P[ \norm{\bar f - f_0}_n > \delta ]^{1/2} \leq \exp( - (\delta^2
  - \delta)^2 / (8 \sigma^2)) \leq ( \exp(-C_2 (\log n)^4)) =
  o(h^2)$ if one chooses $\delta = \log n$. Now,
  using~\eqref{eq:deviaB1} with this choice of $\delta$ and $z = (\log
  n)^{1 + \beta/2}$ we have also $P_n[ \mathcal Z_1(z,
  \delta)^\complement]^{1/2} \leq \exp( -C_3 (\log n)^2) =
  o(h^2)$. This concludes the proof of the first upper bound of
  Theorem~\ref{thm:least_sq}.

  To prove the upper bound for the integrated norm $\norm{\cdot}$
  instead of the empirical norm $\norm{\cdot}_n$, we decompose
  $\norm{\bar f - f_0}^2 = A_1 + A_2$ where
  \begin{equation*}
    A_1 := \norm{\bar f - f_0}^2 - 8 ( \norm{\bar f - f_0}_n^2
    + \pen(\bar f)) \text{ and } A_2 := 8 ( \norm{\bar f - f_0}_n^2 +
    \pen(\bar f)).
  \end{equation*}
  The first part of Theorem~\ref{thm:least_sq} provides
  \begin{equation*}
    E^n[A_2] \leq C_1 ( 1 + |f_0|_{\mathcal F}^\alpha) n^{-2 / (2 +
      \beta)}.
  \end{equation*}
  Recall that we assumed that $\norm{\bar f - f_0}_\infty \leq Q$
  a.s. for the second part of the Theorem. To handle $A_1$, we use the
  following Lemma.
  \begin{lemma}
    \label{lem:devia2}
    Let $(\mathcal F, |\cdot|_{\mathcal F})$ and $h$ satisfy the same
    assumptions as in Theorem~\ref{thm:least_sq}. Define $\mathcal F_Q
    := \{ f \in \mathcal F : \norm{f - f_0}_\infty \leq Q \}$. We can
    find constants $z_0, D_0 > 0$ such that for any $z \geq
    z_0$\textup:
    \begin{align*}
      P_X^n \big[ \exists f \in \mathcal F_Q : \norm{f - f_0}^2 &- 8
      (\norm{f - f_0}_n^2 + \pen(f)) \geq 10 z h^2 \big] \\
      &\leq \exp \big( - D_0 n h^2 z \big),
    \end{align*}
    where $z_0$ and $D_0$ are constants depending on $a, \alpha,
    \beta$ and $Q$.
  \end{lemma}
  The proof of Lemma~\ref{lem:devia2} is given in
  Section~\ref{sec:lemmas_proofs}. Using together
  Lemma~\ref{lem:devia2} and the fact that $A_1 \leq Q^2$ a.s., we
  have by a decomposition over the union of $\{ A_1 \geq 10 z_0 h^2
  \}$ and $\{ A_1 < 10 z_0 h^2 \}$:
  \begin{equation*}
    E^n [A_1] \leq 10 z_0 h^2 + o(h^2).
  \end{equation*}
  This concludes the proof of Theorem~\ref{thm:least_sq}.
\end{proof}

\begin{figure}[htbp]
  \begin{tikzpicture}[scale=2]
    \draw[thick] (0,0) node[anchor=south] {$f_0$} circle (1); %
    \draw (0,0) -- (1,0) node[anchor=west] {$f_1$}; %
    \draw (0,0) -- (40:1cm) node[anchor=south west]{$f_{M-1}$}; %
    \draw (0,0) -- (150:1cm) node[anchor=south east] {$f_3$}; %
    \draw (0,0) -- (190:1cm) node[anchor=east]{$f_{2}$}; %
    \draw (0,0) -- (290:1cm) ; %
    \draw[<->, very thick] (290:1cm) -- (290:0.6cm) node[anchor=east]
    {$f_M$} node[pos=0.5, right] {$h$}; %
    \draw[mark=x] (50:1cm) ;
  \end{tikzpicture}
  \caption{Example of a setup in which ERM performs badly. The set
    $F(\Lambda) = \{f_1, \ldots, f_M \}$ is the dictionary from which
    we want to mimic the best element and $f_0$ is the regression
    function.}
  \label{fig:badsetup}
\end{figure}
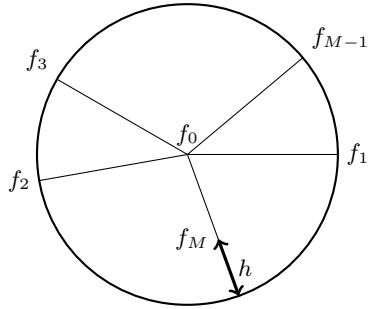

\begin{proof}[Proof of Theorem \ref{TheoWeaknessERMRegression}]
  We consider a random variable $X$ uniformly distributed on $[0,1]$
  and its dyadic representation:
  \begin{equation}
    \label{EquaDyadicRegression}
    X = \sum_{k = 1}^{+\infty} X^{(k)} 2^{-k},
  \end{equation}
  where $(X^{(k)} : k \geq 1)$ is a sequence of i.i.d. random
  variables following a Bernoulli $\cB(1/2,1)$ with parameter $1/2$.
  The random variable $X$ is the design of the regression model worked
  out here. For the regression function we take
  \begin{equation}
    \label{FunctionBasisRegression}
    f_0(x) =
    \begin{cases}
      \; 2h &\text{ if } x^{(M)} = 1 \\
      \; h & \text{ if } x^{(M)} = 0,
    \end{cases}
  \end{equation}
  where $x$ has the dyadic decomposition $x=\sum_{k \geq 1}
  x^{(k)}2^{-k}$ where $x^{(k)} \in \{ 0, 1 \}$ and
  \begin{equation*}
    h=\frac{C}{4}\sqrt{\frac{\log M}{n}}.
  \end{equation*}
  We consider the dictionary of functions $F_M = \{f_1, \ldots, f_M\}$
  \begin{equation}
    \label{FunctionBasisRegression}
    f_j(x) = 2x^{(j)}-1, \quad \forall j\in\{1,\ldots,M\},
  \end{equation}
  where again $(x^{(j)} : j \geq 1)$ is the dyadic decomposition of $x
  \in [0,1]$. The dictionary $F_M$ is chosen so that we have, for any
  $j \in \{ 1, \ldots ,M-1 \}$
  \begin{equation*}
    \| f_j - f_0 \|_{L^2([0,1])}^2 = \frac{5 h^2}{2} + 1 \;\text{ and }\;
    \|f_M - f_0 \|_{L^2([0,1])}^2 = \frac{5h^2}{2} - h + 1.
  \end{equation*}
  Thus, we have
  \begin{equation*}
    \min_{j=1,\ldots,M} \|f_j - f_0 \|_{L^2([0,1])}^2 = \|f_M - f_0
    \|_{L^2([0,1])}^2 = \frac{5h^2}{2} -h + 1.
  \end{equation*}
  This geometrical setup for $F(\Lambda)$, which is a unfavourable
  setup for the ERM, is represented in Figure~\ref{fig:badsetup}. For
  \begin{equation*}
    \hat{f}_n := \tilde{f}_n^{\rm PERM} \in \argmin_{f \in F_M}
    \big(R_n(f) + \pen(f) \big),
  \end{equation*}
  where we take $R_n(f) = \frac{1}{n} \sum_{i=1}^n (Y_i-f(X_i))^2 =\|
  Y - f \|^2_n$, we have
  \begin{equation}
    \label{InegGaussian}
    E \|\hat{f}_n - f_0 \|_{L^2([0,1])}^2 =
    \min_{j=1,\ldots,M} \|f_j - f_0 \|_{L^2([0,1])}^2 + h
    P[\hat{f}_n\neq f_M].
  \end{equation}
  Now, we upper bound $P[ \hat{f}_n= f_M]$. If we define
  \begin{equation*}
    N_j := \frac{1}{\sqrt{n}} \sum_{i=1}^n\zeta_i^{(j)}
    \varepsilon_i \text{ and } \zeta_i^{(j)} := 2X_i^{(j)}-1,
  \end{equation*}
  we have by the definition of $h$ and since $\zeta_i^{(j)} \in \{ -1,
  1\}$:
  \begin{align*}
    \frac{\sqrt{n}}{2 \sigma} (\norm{Y - f_M}_n^2 &- \norm{Y -
      f_j}_n^2) \\
    & = N_j - N_M + \frac{h}{2 \sigma \sqrt{n}} \sum_{i=1}^n
    (\zeta_i^{(j)} \zeta_i^{(M)} + 3(\zeta_i^{(j)} - \zeta_i^{(M)}) -
    1) \\
    &\geq N_j - N_M - \frac{4C}{\sigma} \sqrt{\log M}.
  \end{align*}
  This entails, for $\bar N_{M-1} := \max_{1 \leq j \leq N-1} N_j$,
  that
  \begin{align*}
    P[ \hat{f}_n= f_M] &= P \Big[ \bigcap_{j=1}^{M-1} \Big\{ \norm{Y -
      f_M}_n^2 - \norm{Y - f_j}_n^2 \leq \pen(f_j) - \pen(f_M) \Big\}
    \Big] \\
    &\leq P\Big[ N_M \geq \bar N_{M-1} - \frac{6C}{\sigma} \sqrt{\log
      M} \Big].
  \end{align*}
  It is easy to check that $N_1, \ldots, N_M$ are $M$ normalized
  standard gaussian random variables uncorrelated (but dependent). We
  denote by $\boldsymbol{\zeta}$ the family of Rademacher variables
  $(\zeta_i^{(j)} : i=1,\ldots,n ; j=1,\ldots,M)$. We have for any
  $6C/\sigma <\gamma< (2\sqrt{2}c^*)^{-1}$ ($c^*$ is the ``Sudakov
  constant'', see Theorem~\ref{TheoSudakov}),
  \begin{align}
    \label{EquaSudakov}
    P[\hat{f}_n = f_M] &\leq E \Big[ P\Big( N_M \geq \bar N_{M-1} -
    \frac{6C}{\sigma}\sqrt{\log M} \Big| \boldsymbol{\zeta} \Big)
    \Big] \nonumber \\
    &\leq P \big[ N_M \geq - \gamma \sqrt{\log M}
    +  E(\bar N_{M-1} | \boldsymbol{\zeta} ) \big] \\
    &+ E \Big[ P\Big\{ E( \bar N_{M-1} | \boldsymbol{\zeta} ) - \bar
    N_{M-1} \geq (\gamma - \frac{6C}{\sigma}) \sqrt{\log M} \Big|
    \boldsymbol{\zeta} \Big\} \Big]. \nonumber
  \end{align}
  Conditionally to $\boldsymbol{\zeta}$, the vector
  $(N_1,\ldots,N_{M-1})$ is a linear transform of the Gaussian vector
  $(\varepsilon_1, \ldots, \varepsilon_n)$. Hence, conditionally to
  $\boldsymbol{\zeta}$, $(N_1,\ldots,N_{M-1})$ is a gaussian
  vector. Thus, we can use a standard deviation result for the
  supremum of Gaussian random vectors (see for
  instance~\cite{massart03}, Chapter~3.2.4), which leads to the
  following inequality for the second term of the RHS
  in~\eqref{EquaSudakov}:
  \begin{align*}
    P \Big\{ E( \bar N_{M-1} | \boldsymbol{\zeta} ) - \bar N_{M-1}
    \geq (\gamma &- \frac{6C}{\sigma}) \sqrt{\log M} \Big|
    \boldsymbol{\zeta}
    \Big\} \\
    &\leq \exp(-(3C/\sigma-\gamma/2)^2\log M).
  \end{align*}
  Remark that we used $E[ N_j^2 | \boldsymbol{\zeta}] = 1$ for any $j
  = 1, \ldots, M-1$. For the first term in the RHS
  of~\eqref{EquaSudakov}, we have
  \begin{align}
    \label{EquaIerTermSudakov}
    P &\Big [N_M \geq - \gamma \sqrt{\log M}
    + E( \bar N_{M-1} | \boldsymbol{\zeta} ) \Big] \nonumber\\
    &\leq P \Big[N_M \geq - 2 \gamma \sqrt{\log M}
    + E(\bar N_{M-1}) \Big] \\
    &+P \Big[ - \gamma\sqrt{\log M} + E(\bar N_{M-1}) \geq E(\bar
    N_{M-1} | \boldsymbol{\zeta}) \Big]. \nonumber
  \end{align}
  Next, we use Sudakov's Theorem (cf. Theorem \ref{TheoSudakov} in
  Appendix~\ref{sec:appendix_proba}) to lower bound $E( \bar
  N_{M-1})$. Since $(N_1,\ldots,N_{M-1})$ is, conditionally to
  $\boldsymbol{\zeta}$, a Gaussian vector and since for any $1 \leq j
  \neq k \leq M$ we have
  \begin{equation*}
    E[(N_k-N_j)^2 | \boldsymbol{\zeta}] = \frac{1}{n}
    \sum_{i=1}^n (\zeta_i^{(k)} - \zeta_i^{(j)})^2
  \end{equation*}
  then, according to Sudakov's minoration
  (cf. Theorem~\ref{TheoSudakov} in the Appendix), there exits an
  absolute constant $c^* > 0$ such that
  \begin{equation*}
    c^* E[\bar N_{M-1} | \boldsymbol{\zeta}] \geq
    \min_{1 \leq j \neq k \leq M-1} \Big(\frac{1}{n}\sum_{i=1}^n
    (\zeta_i^{(k)} - \zeta_i^{(j)})^2\Big)^{1/2} \sqrt{\log M}.
  \end{equation*}
  Thus, we have
  \begin{align*}
    \label{EquaSudak3}
    c^* E[\bar N_{M-1}] &\geq E\Big[ \min_{j \neq k} \Big(\frac{1}{n}
    \sum_{i=1}^n (\zeta_i^{(k)} - \zeta_i^{(j)})^2
    \Big)^{1/2} \Big] \sqrt{\log M} \\
    &\geq \sqrt{2} \Big(1 - E\Big[ \max_{j\neq k} \frac{1}{n}
    \sum_{i=1}^n \zeta_i^{(k)} \zeta_i^{(j)} \Big] \Big) \sqrt{\log M},
  \end{align*}
  where we used the fact that $\sqrt{x} \geq x/\sqrt{2}, \forall x \in
  [0,2]$.
  Besides, using Hoeffding's inequality we have $E[\exp(s
  \xi^{(j,k)})] \leq \exp(s^2/(2n))$ for any $s > 0$, where
  $\xi^{(j,k)} := n^{-1} \sum_{i=1}^n \zeta_i^{(k)} \zeta_i^{(j)}$.
  Then, using a maximal inequality (cf.  Theorem~\ref{TheoMaxConcIneq}
  in Appendix~\ref{sec:appendix_proba}) and since $n^{-1}
  \log[(M-1)(M-2)] \leq 1/4$, we have
  \begin{equation}
    \label{EquaSudakFinal}
    E\Big[\max_{j\neq k} \frac{1}{n} \sum_{i=1}^n
    \zeta_i^{(k)} \zeta_i^{(j)} \Big] \leq
    \Big(\frac{1}{n} \log[(M-1)(M-2)] \Big)^{1/2} \leq
    \frac{1}{2}.
  \end{equation}
  This entails
  \begin{equation*}
    c^* E[ \bar N_{M-1} ] \geq \Big(\frac{\log M}{2} \Big)^{1/2}.
  \end{equation*}
  Thus, using this inequality in the first RHS
  of~\eqref{EquaIerTermSudakov} and the usual inequality on the tail
  of a Gaussian random variable ($N_M$ is standard Gaussian), we
  obtain:
  \begin{align}
    \label{EquaFirstTerm}
    P\Big[N_M \geq &-2\gamma \sqrt{\log M} + E(\bar N_{M-1}) \Big]
    \leq P\Big[ N_M \geq ((c^*\sqrt{2})^{-1}-2\gamma)
    \sqrt{\log M}\Big]\nonumber\\
    &\leq \mathbb{P}\Big[N_M \geq ((c^*\sqrt{2})^{-1}-2\gamma)
    \sqrt{\log
      M}\Big]\\
    &\leq \exp\Big(-((c^*\sqrt{2})^{-1}-2\gamma)^2(\log
    M)/2\Big).\nonumber
  \end{align}
  Remark that we used $2\sqrt{2}c^* \gamma < 1$. For the second term
  in (\ref{EquaIerTermSudakov}), we apply the concentration inequality
  of Theorem \ref{TheoEinmahlMasson} to the non-negative random
  variable $E[\bar N_{M-1}|\boldsymbol{\zeta}]$. We first have to
  control the second moment of this variable. We know that,
  conditionally to $\boldsymbol{\zeta}$,
  $N_j|\boldsymbol{\zeta}\sim\cN(0,1)$ thus,
  $N_j|\boldsymbol{\zeta}\in L_{\psi_2}$ (for more details on Orlicz
  norm, we refer the reader to~\cite{vdVW:96}). Thus,
  \begin{equation*}
    \norm{\max_{1\leq j\leq M-1} N_j|\boldsymbol{\zeta}}_{\psi_2}\leq K
    \psi_2^{-1}(M)\max_{1\leq j\leq M-1}\norm{N_j|\boldsymbol{\zeta}}_{\psi_2}
  \end{equation*}
  (cf. Lemma 2.2.2 in \cite{vdVW:96}). Since
  $\norm{N_j|\boldsymbol{\zeta}}_{\psi_2}^2=1$, we have $\norm{\max_{1\leq j\leq M-1}
    N_j|\boldsymbol{\zeta}}_{\psi_2}\leq K \sqrt{\log M}$. In particular, we have
  $E\big[\max_{1\leq j\leq M-1} N_j^2|\boldsymbol{\zeta}\big]\leq
  K\log M$ and so $E\big(E[\bar
  N_{M-1}|\boldsymbol{\zeta}]\big)^2\leq K\log M$. Theorem
  \ref{TheoEinmahlMasson} provides
  \begin{equation}
    \label{SecondTermEquaSuda}
    P\Big[ -\gamma\sqrt{\log
      M}+E[\bar N_{M-1}]\geq E[\bar N_{M-1}|\boldsymbol{\zeta}]\Big]\leq
    \exp(-\gamma^2/c_0),
  \end{equation}
  where $c_0$ is an absolute constant.

Finally, combining (\ref{EquaSudakov}), (\ref{EquaFirstTerm}),
(\ref{EquaIerTermSudakov}), (\ref{SecondTermEquaSuda}) in the initial
inequality (\ref{EquaSudakov}), we obtain
\begin{align*}
P[\hat{f}_n= f_M] &\leq \exp(-(3C/\sigma-\gamma)^2\log M)\\
&+
\exp\Big(-((c^*\sqrt{2})^{-1}-2\gamma)^2(\log M)/2\Big)+
\exp(-\gamma^2/c_0).
\end{align*}
Take $\gamma=(12\sqrt{2}c^*)^{-1}$. It is easy to find an integer $M_0(\sigma)$ depending only on $\sigma$ such that for any $M\geq M_0$, we have $P[\hat{f}_n= f_M]\leq c_1<1$, where $c_1$ is an absolute constant.
We complete the proof by using this last result in
(\ref{InegGaussian}).
\end{proof}

\begin{proof}[Proof of Theorem~\ref{thm:oracle}]
  We recall that we have a dictionary (set of functions) $F(\Lambda)$
  of cardinality $M$ such that $\norm{f_\lambda - f_0}_\infty \leq Q$
  for all $\lambda \in \Lambda$. Let us define the risk
  \begin{equation*}
    R(f) := E[(Y - f(X))^2]
  \end{equation*}
  and the linearized risk over $F(\Lambda)$, given by
  \begin{equation*}
    \mathsf R(\theta) := \sum_{\lambda \in \Lambda} \theta_\lambda
    R(f_\lambda)
  \end{equation*}
  for $\theta \in \Theta$, where we recall that
  \begin{equation*}
    \Theta := \{ \theta \in \mathbf R^{|\Lambda|} ; \theta_\lambda
    \geq 0,\; \sum_{\lambda \in \Lambda} \theta_\lambda = 1 \}.
  \end{equation*}
  We denote by $R_{n}(f)$ the empirical risk of $f$ over the sample
  $D_{n}$, which is given by
  \begin{equation*}
    R_{n}(f) := \frac{1}{n} \sum_{i=1}^n (Y_i - f(X_i))^2,
  \end{equation*}
  and we define similarly the linearized empirical risk
  \begin{equation*}
    \mathsf R_{n}(\theta) := \sum_{\lambda \in \Lambda}
    \theta_\lambda R_{n}(f_\lambda).
  \end{equation*}
  The excess risk of a function $f$ is given by $R(f) - R(f_0) =
  \norm{f - f_0}^2$. By convexity of the risk, the aggregate $\hat
  {\mathsf f}= \sum_{\lambda \in \Lambda} \hat \theta_\lambda
  f_\lambda$ defined in (\ref{eq:aggregate}), satisfies, for any $a >
  0$,
  \begin{align*}
    R(\hat {\mathsf f}) - R(f_0) &\leq \mathsf R(\hat \theta) - R(f_0) \\
    &\leq (1 + a) (\mathsf R_{n}(\hat \theta) - R_{n}(f_0)) \\
    &+ \mathsf R(\hat \theta) - R(f_0) - (1 + a) (\mathsf R_{n}(\hat
    \theta) - R_{n}(f_0)),
  \end{align*}
  where it is easy to see that the Gibbs weights $\hat \theta = (\hat
  \theta_\lambda)_{\lambda \in \Lambda} = (\hat
  \theta(f_\lambda))_{\lambda \in \Lambda}$ are the unique solution to
  the minimization problem
  \begin{equation*}
    \min_{\theta \in \Theta} \Big\{ \mathsf R_{n}(\theta) +
    \frac{T}{ n} \sum_{\lambda \in \Lambda} \theta_\lambda \log
    \theta_\lambda \Big\},
  \end{equation*}
  where $T$ is the temperature parameter, see~\eqref{eq:weights}, and
  where we use the convention $0 \log 0 = 0$. Let $\hat \lambda$ be
  such that $f_{\hat \lambda}$ is the ERM in $F(\Lambda)$, namely
  \begin{equation*}
    R_{n}(f_{\hat \lambda}) := \min_{\lambda \in \Lambda}
    R_{n}(f_\lambda).
  \end{equation*}
  Since
  \begin{equation*}
    \sum_{\lambda \in \Lambda} \hat \theta_\lambda \log \Big( \frac{\hat
      \theta_\lambda}{1 / |\Lambda|} \Big) = K(\hat \theta | u) \geq 0
  \end{equation*}
  where $K(\hat \theta | u)$ denotes the Kullback-Leibler divergence
  between the weights $\hat \theta$ and the uniform weights $u := (1 /
  |\Lambda|)_{\lambda \in \Lambda}$, we have
  \begin{align*}
    \mathsf R_{n}(\hat \theta) &\leq \mathsf R_{n}(\hat \theta) +
    \frac{T}{ n} K(\hat \theta | u) \\
    &= \mathsf R_{n}(\hat \theta) + \sum_{\lambda \in \Lambda} \hat
    \theta_\lambda \log \hat \theta_\lambda + \frac{T\log |\Lambda|}{
      n} \\
    &\leq \mathsf R_{n}(e_{\hat \lambda}) + \frac{T\log |\Lambda|}{
      n} = R_{n}(f_{\hat \lambda}) + \frac{T\log |\Lambda|}{n},
  \end{align*}
  where $e_\lambda \in \Theta$ is the vector with $1$ for the
  $\lambda$-th coordinate and $0$ elsewhere. This gives
  \begin{align*}
    R(\hat {\mathsf f}) - R(f_0) &\leq (1 + a) \min_{\lambda \in \Lambda}
    (R_{n}(f_\lambda) - R_{n}(f_0))+ (1 + a)
    \frac{T\log |\Lambda|}{ n} \\
    &+ \mathsf R(\hat \theta) - R(f_0) - (1 + a) (\mathsf R_{n}(\hat
    \theta) - R_{n}(f_0)),
  \end{align*}
  and consequently
  \begin{align*}
    E \norm{\hat {\mathsf f} - f_0}^2 &\leq (1 + a) \min_{\lambda \in
      \Lambda} \norm{f_\lambda - f_0}^2 + (1 + a) \frac{T\log
      |\Lambda|}{n} \\
    &+ E[ \mathsf R(\hat \theta) - R(f_0) - (1 + a) (\mathsf
    R_{n}(\hat \theta) - R_{n}(f_0)) ].
  \end{align*}
  Since $\mathsf R(\cdot)$ and $\mathsf R_{n}$ are linear on
  $\Theta$, we have
  \begin{align*}
    \mathsf R(\hat \theta) - R(f_0) &- (1 + a) (\mathsf R_{n}(\hat
    \theta) - R_{n}(f_0)) \\
    &\leq \max_{f \in F(\Lambda)} ( R(f) - R(f_0) - (1 + a)
    (R_{n}(f) - R_{n}(f_0)) ).
  \end{align*}
  Thus, we have
  \begin{equation}\label{eq:Main0}
    E \norm{\hat {\mathsf f} - f_0}^2 \leq (1 + a)
    \min_{\lambda \in \Lambda} \norm{f_\lambda - f_0}^2 + (1 + a)
    \frac{\log |\Lambda|}{T n} + \mathcal R_n,
  \end{equation}
  where $\mathcal R_n := E [ \max_{f \in F(\Lambda)} \{ R(f) - R(f_0)
  - (1 + a) (R_{n}(f) - R_{n}(f_0)) \} ] $. Now, we upper bound
  $\mathcal R_n$. Introduce the random variables
  \begin{align*}
    \tilde{Z}_i(f) &:= (f(X_i) - f_0(X_i))^2 + 2 \sigma \varepsilon_i
    I( |\varepsilon_i| \leq K) (f_0(X_i) - f(X_i)), \\
    \bar Z_i(f) &:= 2 \sigma \varepsilon_i I(|\varepsilon_i| > K)
    (f_0(X_i) - f(X_i)),
  \end{align*}
  and the two following processes indexed by $f \in F(\Lambda)$:
  \begin{equation*}
    \tilde{\zeta}(f) := \frac{1}{n}\sum_{i=1}^n \Big(
    E[\tilde{Z}_i(f)] - (1+a) \tilde{Z}_i(f) \Big) \text{ and }
    \bar{\zeta}(f) := \frac{1+a}{n} \sum_{i=1}^n\bar{Z}_i(f).
  \end{equation*}
  We use the symmetry of $\varepsilon$ to get
  \begin{equation*}
    \mathcal R_n \leq E \Big[ \max_{f \in F(\Lambda)}
    \tilde{\zeta}(f) \Big] + E \Big[ \max_{f \in F(\Lambda)}
    \bar{\zeta}(f) \Big].
  \end{equation*}
  First, we upper bound $E[ \max_{f \in F(\Lambda)}
  \tilde{\zeta}(f)]$. The random variable $\tilde{\zeta}(f)$ is
  bounded and satisfies the following Bernstein's type condition
  (see~\cite{BM:06}): $\forall f \in F(\Lambda), E [
  \tilde{\zeta}(f)^2] \leq (Q^2 + 4 \sigma^2) E[\tilde{\zeta}(f)]$. We
  apply the union bound and the Bernstein's inequality
  (cf. \cite{vdVW:96}) to get, for any $\delta>0$,
  \begin{align*}
    P \Big[\max_{f\in F(\Lambda)} \tilde{\zeta}(f) \geq \delta \Big]
    &\leq \sum_{f\in F(\Lambda)} P\Big[ \frac{1}{n}\sum_{i=1}^n
    E[\tilde{Z}_i(f)] - \tilde{Z}_i(f) \geq
    \frac{\delta + a E[\tilde{Z}_i(f)] }{1+a} \Big] \\
    &\leq M \exp(-C n \delta),
  \end{align*}
  where $C := a [8 (Q^2 + \sigma^2 (1 + a)^2 + (4Q / 3)(1 + a)(Q +
  2K)]^{-1}$. Hence, a direct computation gives
  \begin{equation}
    \label{eq:I1}
    E\Big[ \max_{f\in F(\Lambda)} \tilde{\zeta}(f) \Big] \leq
    \frac{4 \log M}{C n}.
  \end{equation}
  Now, we upper bound $E [\max_{f\in F(\Lambda)}\bar{\zeta}(f) ]$. We
  have
  \begin{align}
    \label{eq:I2}
    \nonumber E \Big[ \max_{f\in F(\Lambda)} \bar{\zeta}(f) \Big]
    &\leq 4 Q (1 + a) E \big[ |\varepsilon| I(|\varepsilon| > K) \big] \\
    &\leq 4 Q (1 + a) \sigma P (|\varepsilon|>K)^{1/2} \\
    &\leq 4Q(1+a) \sigma \exp(-K^2 / (2 b_\varepsilon^2)).
  \end{align}
  Finally, combining equations \eqref{eq:Main0},~\eqref{eq:I1})
  and~\eqref{eq:I2} with $K = b_\varepsilon \sqrt{2 \log n}$,
  concludes the proof of Theorem~\ref{thm:oracle}.
\end{proof}




\section{Proofs of the lemmas}
\label{sec:lemmas_proofs}

\begin{proof}[Proof of Lemma~\ref{lem:logtrick}]
  Since $\beta \in (0, 2)$ we have $\alpha > 2 \beta / (\beta + 2) >
  \beta/2$. Thus, inequality~\eqref{eq:logtrick} gives
  \begin{align*}
    \log(r^2 + h^2 I^\alpha) &\leq \log(\varepsilon) + (1 -
    \frac{\beta}{2}) \log(r) - (1 - \frac{\beta}{2\alpha}) \log(r^2)
    \\
    & - \frac{\beta}{\alpha} \log(h) + (1 - \frac{\beta}{2\alpha})
    \log(r^2) + \frac{\beta}{2\alpha} \log(h^2 I^\alpha) \\
    &\leq \log(\varepsilon) + (\frac{\beta}{\alpha} - 1 -
    \frac{\beta}{2}) \log(r) - \frac{\beta}{\alpha} \log(h) + \log(
    r^2 + h^2 I^\alpha)
  \end{align*}
  and consequently
  \begin{equation*}
    r^{1 + \beta / 2 - \beta/\alpha} \leq \varepsilon h^{-\beta/\alpha}
  \end{equation*}
  which entails $r \leq ( \varepsilon^{\alpha} h^{-\beta} )^{2 / (2\alpha
    + \alpha \beta - 2 \beta)}$. Now, using this inequality together
  with $h^2 I^\alpha \leq \varepsilon\, r^{1 - \beta / 2} I^{\beta / 2}$
  provides the upper bound for $I$. The last inequality easily follows.
\end{proof}

\begin{proof}[Proof of Lemma~\ref{lem:devia2}]
  [The proof consists of a \emph{peeling} of $\mathcal F$ into
  subspaces with complexity controlled by Assumption~$(C_\beta)$ and
  the use of Bernstein's inequality.] Let us denote for short
  $\mathcal F$ instead of $\mathcal F_Q$. Since $\bar f \in \mathcal
  F$, we have
  \begin{align*}
    P \big[ \norm{\bar f &- f_0}^2 - 8 (\norm{\bar f - f_0}_n^2 +
    \pen(\bar f)) \geq 10 z h^2 \big] \\
    &\leq P \big[ \exists f \in \mathcal F : \norm{f - f_0}^2 - 8
    ( \norm{f - f_0}_n^2 + \pen(f) ) \geq 10 z h^2 \big] \\
    &\leq P[A_1] + \sum_{k \geq 2} P[A_k],
  \end{align*}
  where
  \begin{align*}
    A_1 := \big\{ \exists f &\in \mathcal F,\;\pen(f) \leq 2^{\alpha /
      \beta} h^2 : \\
    &\norm{f - f_0}^2 - 8 (\norm{f - f_0}_n^2 + \pen(f) ) \geq 10 z
    h^2 \big\}
  \end{align*}
  and for $k \geq 2$,
  \begin{align*}
    A_k := \big\{ \exists f \in \mathcal F,\; &2^{\alpha (k-1) /
      \beta} h^2 < \pen(f) \leq 2^{\alpha k / \beta} h^2 : \\
    &\norm{f - f_0}^2 - 8 (\norm{f - f_0}_n^2 + \pen(f) ) \geq 10 z
    h^2 \big\}.
  \end{align*}
  Hence, since $z \geq z_0 \geq 1$ and $\alpha / \beta = 2 / (\beta +
  2) > 1/2$ since $\beta < 2$, we have $P[A_k] \leq P_k$ for any $k
  \geq 1$, where
  \begin{align*}
    P_k := P \big[ \exists f \in \mathcal F,\; &\pen(f) \leq 2^{\alpha
      k / \beta} h^2 : \\
    &\norm{f - f_0}^2 - 8 \norm{f - f_0}_n^2 \geq 2 z h^2 + 4
    2^{\alpha k / \beta} h^2 \big].
  \end{align*}
  Now, let $F(\delta, k)$ be a minimal $\delta$-covering for the norm
  $\norm{\cdot}_\infty$ of the set
  \begin{equation*}
    \{ f \in \mathcal F : \pen(f) \leq 2^{\alpha k / \beta} h^2 \} =
    \{ f \in \mathcal F : |f|_{\mathcal F} \leq 2^{k /\beta} \},
  \end{equation*}
  where we recall that $\pen(f) = h^2 |f|_{\mathcal
    F}^\alpha$. Assumption~$(C_\beta)$ entails
  \begin{equation}
    \label{eq:covering1}
    | F(\delta, k) | \leq \exp ( D 2^{k} \delta^{-\beta} ).
  \end{equation}
  Since for any $f_1, f_2 \in \mathcal F$ such that $\norm{f_1 -
    f_2}_\infty \leq \delta$, we have
  \begin{equation*}
    \norm{f_1 - f_0}^2 \leq 2\norm{f_2 - f_0}^2 + 2 \delta^2 \quad
    \text{ and } \quad 2
    \norm{f_1 - f_0}_n^2 \geq 2\norm{f_2 - f_0}_n^2 - 2 \delta^2,
  \end{equation*}
  we obtain
  \begin{align*}
    P_k &\leq P \big[ \exists f \in F(\delta, k) : 2 \norm{f - f_0}^2
    - 4 \norm{f - f_0}_n^2 + 6 \delta^2 \geq 2 z h^2 + 4 2^{\alpha k /
      \beta} h^2 \big] \\
    &\leq \sum_{f\in F(\delta, k)} \times P \big[ \norm{f - f_0}^2 - \norm{f -
      f_0}_n^2 \geq t_k(z) \big],
  \end{align*}
  where $t_k(z) := z h^2 / 2 + 2^{\alpha k / \beta} h^2 - 3 \delta^2 /
  2 + \norm{f - f_0}^2 / 2$. Let $f \in F(\delta, k)$ be fixed. We
  introduce the random variables $U_i := (f(X_i) - f_0(X_i))^2$, so
  that $\norm{f - f_0}_n^2 = \sum_{i=1}^n U_i / n$ and $E[U_1] =
  \norm{f - f_0}^2$. Note that the $U_i$ are independent, such that $0
  \leq U_i \leq Q^2$, and $\var [U_1] \leq E [U_1^2] \leq Q^2 E [U_1]
  \leq Q^2 \norm{f - f_0}^2$. Hence, if $t_k(z) \geq \norm{f - f_0}^2
  / 2$, Bernstein's inequality entails
  \begin{align*}
    P \big[ \norm{f - f_0}^2 &- \norm{f - f_0}_n^2 \geq t_k(z) \big]
    = P \Big[ \sum_{i=1}^n (U_i - E [U_1]) \geq n t_k(z) \Big] \\
    &\leq \exp \Big( \frac{-n t_k(z)^2}{2( Q^2 \norm{f -
        f_0}^2 + Q^2 t_k(z) / 3)} \Big) \\
    &\leq \exp \Big( \frac{-3 n ( z h^2 + 2^{\alpha k / \beta +1} h^2
      - 3 \delta^2 )}{28 Q^2} \Big).
  \end{align*}
  By taking $\delta := (2^{\alpha k / \beta} h^2 / 3)^{1/2}$, we have
  $t_k(z) \geq \norm{f - f_0}^2 / 2$ and \eqref{eq:covering1} becomes
  \begin{equation*}
    | F(\delta, k) | \leq \exp \Big( D_1 n h^2 2^{k(1 - \alpha / 2)}
    \Big),
  \end{equation*}
  where we used~\eqref{eq:bandwidth} and took $D_1 := D 3^{\beta / 2}
  / a^{\beta + 2}$. Hence, for $D_2 := 3 / (28 Q^2)$, we have
  \begin{equation*}
    P_k \leq \exp\Big( D_1 n h^2 2^{k (1 - \alpha / 2)} - D_2 n h^2 (z +
    2^{\alpha k / \beta}) \Big).
  \end{equation*}
  Now, we choose
  \begin{equation*}
    K := \Big[ \frac{\log (\min(D_2 / D_1, 1) / 2)}{(1 - \alpha / 2 - \alpha
      / \beta) \log 2} \Big] + 1,
  \end{equation*}
  where $[x]$ is the integer part of $x$, and where we recall that
  $\alpha > 2 \beta / (\beta + 2)$, so that $1 - \alpha / 2 - \alpha /
  \beta < 0$. The conclusion of the proof follows easily by the
  decomposition $\sum_{k \geq 1} P_k = \sum_{1 \leq k < K} P_k +
  \sum_{k \geq K} P_k$, if $z \geq z_1$ for the choice $z_1 := 2 (
  2^{K \alpha / \beta} - D_1 2^{K(1 - \alpha / 2)} / D_2)$.
\end{proof}

\appendix

\section{Function spaces}
\label{sec:appendix_approximation}

In this section we give precise definitions of the spaces of functions
considered in the paper, and give useful related results. The
definitions and results presented here can be found
in~\cite{triebel06}, in particular in Chapter~5 which is about
anisotropic spaces, anisotropic multiresolutions, and entropy numbers
of the embeddings of such spaces (see Section~5.3.3) that we use in
particular to derive condition $(C_\beta)$, for the anisotropic Besov
space, see Section~\ref{sec:pena_least_squares}.


\subsection{Anisotropic Besov space}

Let $\{ e_1, \ldots, e_d \}$ be the canonical basis of $\mathbb R^d$
and $\bs s = (s_1, \ldots, s_d)$ with $s_i > 0$ be a vector of
directional smoothness, where $s_i$ corresponds to the smoothness in
direction $e_i$. Let us fix $1 \leq p, q \leq \infty$. If $f$ is a
function in $\mathbb R^d$, we define $\Delta_h^k f$ as the
\emph{difference} of order $k \geq 1$ and step $h \in \mathbb R^d$,
given by $\Delta_h^1 f(x) = f(x + h) - f(x)$ and $\Delta_h^k f(x) =
\Delta_h^1(\Delta_h^{k-1}f)(x)$ for any $x \in \mathbb R^d$. We say
that $f \in L^p(\mathbb R^d)$ belongs to the anisotropic Besov space
$B_{p, q}^{\bs s}(\mathbb R^d)$ if the semi-norm
\begin{equation*}
  |f|_{B_{p, q}^{\bs s}(\mathbb R^d)} := \sum_{i=1}^d \Big(
  \int_0^1 (t^{-s_i} \norm{\Delta_{t e_i}^{k_i} f}_{p})^q
  \frac{dt}{t} \Big)^{1/q}
\end{equation*}
is finite (with the usual modifications when $p = \infty$ or $q =
\infty$). We know that the norms
\begin{equation*}
  \norm{f}_{B_{p, q}^{\bs s}} := \norm{f}_p + |f|_{B_{p, q}^{\bs s}}
\end{equation*}
are equivalent for any choice of $k_i > s_i$. An equivalent definition
of the seminorm can be given using the directional differences and the
anisotropic distance, see Theorem~5.8 in~\cite{triebel06}.
Following Section~5.3.3 in~\cite{triebel06}, we can define the
anisotropic Besov space on an arbitrary domain $\Omega \subset \mathbb
R^d$ (think of $\Omega$ as the support of the design $X$) in the
following way. We define $B_{p, q}^{\bs s}(\Omega)$ as the set of all
$f \in L^p(\Omega)$ such that there is $g \in B_{p, q}^{\bs s}(\mathbb
R^d)$ with restriction $g | \Omega$ to $\Omega$ equal to $f$ in
$L^p(\Omega)$. Moreover,
\begin{equation*}
  \norm{f}_{B_{p, q}^{\bs s}(\Omega)} = \inf_{g : g|\Omega = f}
  \norm{g}_{B_{p, q}^{\bs s}(\mathbb R^d)},
\end{equation*}
where the infimum is taken over all $g \in B_{p, q}^{\bs s}(\mathbb
R^d)$ such that $g | \Omega = f$. In an equivalent way, the space
$B_{p, q}^{\bs s}(\Omega)$ can be defined using intrisic
characterisations by differences, see Section~4.1.4
in~\cite{triebel06}, where the idea is, roughly, to restrict the
increments $h$ in the differences $\Delta_h^k$ so that the support of
$\Delta_h^k f$ is included in $\Omega$.

In what follows, we shall remove from the notations the dependence on
$\Omega$, since it is does not affect the definitions and results
below. Moreover, for what we need in this paper, we shall simply take
$\Omega$ as the support of the design $X$. Several explicit particular
cases for the space $B_{p, q}^{\bs s}$ are of interest. If $\bs s =
(s, \ldots, s)$ for some $s > 0$, then $B_{p, q}^{\bs s}$ is the
standard isotropic Besov space. When $p = q = 2$ and $s = (s_1,
\ldots, s_d)$ has integer coordinates, $B_{2, 2}^{\bs s}$ is the
anisotropic Sobolev space
\begin{equation*}
  B_{2, 2}^{\bs s} = W_2^{\bs s} = \Big\{ f \in L^2 : \sum_{i=1}^d
  \Big\| \frac{\partial^{s_i} f}{\partial x_i^{s_i}} \Big\|_2 < \infty
  \Big\}.
\end{equation*}
If $\bs s$ has non-integer coordinates, then $B_{2, 2}^{\bs s}$ is the
anisotropic Bessel-potential space
\begin{equation*}
  H^{\bs s} = \Big\{ f \in L^2 : \sum_{i=1}^d \Big\| (1 +
  |\xi_i|^2)^{s_i/2} \hat f(\xi) \Big\|_2 < \infty \Big\}.
\end{equation*}

The results described in the next section are direct consequences of
the transference method, see Section~5.3 in~\cite{triebel06}. Roughly,
the idea is to transfer problems for anisotropic spaces via sequence
space (one can think of sequence of wavelet coefficients for instance)
to isotropic spaces. This technique allows to prove the statements
below. Note that another technique of proof based on replicant coding
can be used, see~\cite{kerk_picard_replicant_03}. This is commented
below.

\subsection{Embeddings and entropy numbers}



Let us first mention the following obvious embedding, which is useful
for the proof of adaptive upper bound (see
Section~\ref{sec:derive_adaptive}). If $0 < \bs s_1 \leq \bs s_0$
coordinatewise, that is $0 < s_{1, i} \leq s_{0, i}$ for any $i \in \{
1, \ldots, d \}$, we have
\begin{equation}
  \label{eq:anisotropic_embedding}
  B_{p, q}^{\bs s_0} \subset B_{p, q}^{\bs s_1}.
\end{equation}
This simply follows from the fact that $B_{p, q}^{\bs s} =
\cap_{i=1}^d B_{p, q, i}^{s_i}$, where $B_{p, q, i}^{s_i}$ is the
corresponding Besov space in the $i$-th direction of coordinates, with
norm $L^p$ extended to the other $d-1$ directions (see Remark~5.7 in
\cite{triebel06}) together with the standard embedding for the
isotropic Besov space.


As we mentioned below, Assumption~$(C_\beta)$ (see
Section~\ref{sec:pena_least_squares}) is satisfied for barely all
smoothness spaces considered in nonparametric literature. In
particular, if $\mathcal F = B_{p,q}^{\bs s}$ is the anisotropic Besov
space defined above, $(C_\beta)$ is satisfied: it is a consequence of
a more general Theorem (see Theorem~5.30 in \cite{triebel06})
concerning the entropy numbers of embeddings (see Definition~1.87 in
\cite{triebel06}). Here, we only give a simplified version of this
Theorem, which is sufficient to derive $(C_\beta)$. Indeed, if one
takes $\bs s_0 = \bs s$, $p_0 = p$, $q_0 = q$ and $\bs s_1 = 0$, $p_0
= \infty$, $q_0 = \infty$ in Theorem~5.30 from \cite{triebel06}, we
obtain the following
\begin{theorem}
  \label{thm:anisotropic_entropy}
  Let $1 \leq p, q \leq \infty$ and $\bs s = (s_1, \ldots, s_d)$ where
  $s_i > 0$\textup, and let $\bs {\bar s}$ be the harmonic mean of
  $\bs s$ \textup(see~\eqref{eq:harmonic_mean}\textup). Whenever $\bs
  {\bar s} > d / p$\textup, we have
  \begin{equation*}
    B_{p, q}^{\bs s} \subset C(\Omega),
  \end{equation*}
  where $C(\Omega)$ is the set of continuous functions on
  $\Omega$\textup, and for any $\delta > 0$\textup, the sup-norm
  entropy of the unit ball of the anisotropic Besov space\textup,
  namely the set
  \begin{equation*}
    U_{p, q}^{\bs s} := \{ f \in B_{p, q}^{\bs s} :
    |f|_{B_{p,q}^{\bs s}} \leq 1 \}
  \end{equation*}
  satisfies
  \begin{equation}
    H_\infty(\delta, U_{p, q}^{\bs s}) \leq D \delta^{-\bs {\bar s} / d},
  \end{equation}
  where $D > 0$ is a constant independent of $\delta$.
\end{theorem}

For the isotropic Sobolev space, Theorem~\ref{thm:anisotropic_entropy}
was obtained in the key paper~\cite{birman_solomjak67} (see
Theorem~5.2 herein), and for the isotropic Besov space, it can be
found, among others, in~\cite{birge_massart00}
and~\cite{kerk_picard_replicant_03}.

\begin{remark}
  A more constructive computation of the entropy of anisotropic Besov
  spaces can be done using the replicant coding approach, which is
  done for Besov bodies in~\cite{kerk_picard_replicant_03}. Using this
  approach together with an anisotropic multiresolution analysis based
  on compactly supported wavelets or atoms, see Section~5.2
  in~\cite{triebel06}, we can obtain a direct computation of the
  entropy. The idea is to do a quantization of the wavelet
  coefficients, and then to code them using a replication of their
  binary representation, and to use 01 as a separator (so that the
  coding is injective). A lower bound for the entropy can be obtained
  as an elegant consequence of Hoeffding's deviation inequality for
  sums of i.i.d. variables and a combinatorial lemma.
\end{remark}

\section{Some probabilistic tools}
\label{sec:appendix_proba}

For the first Theorem we refer to \cite{EM:96}. The two following
Theorems can be found, for instance, in
\cite{massart03,vdVW:96,ledoux_talagrand91}.

\begin{theorem}[Einmahl and Masson (1996)]
  \label{TheoEinmahlMasson}
  Let $Z_1,\ldots,Z_n$ be $n$ independent non-negative random
  variables such that $E[Z_i^2]\leq \sigma^2,\forall i=1, \ldots, n$.
  Then, we have, for any $\delta > 0$,
  \begin{equation*}
    P \Big[\sum_{i=1}^n Z_i - E[Z_i] \leq -n \delta \Big]
    \leq \exp\Big(-\frac{n \delta^2}{2\sigma^2} \Big).
  \end{equation*}
\end{theorem}

\begin{theorem}[Sudakov]
  \label{TheoSudakov}
  There exists an absolute constant $c^*>0$ such that for any integer
  $M$, any centered gaussian vector $X = (X_1,\ldots,X_M)$ in
  $\mathbb{R}^M$, we have,
  \begin{equation*}
    c^* E[\max_{1\leq j\leq M}X_j] \geq \varepsilon \sqrt{\log M},
  \end{equation*}
  where $\varepsilon := \min \Big\{ \sqrt{E[(X_i-X_j)^2]} : i \neq j
  \in \{1, \ldots, M\} \Big\}$.
\end{theorem}

\begin{theorem}[Maximal inequality]
  \label{TheoMaxConcIneq}
  Let $Y_1, \ldots, Y_M$ be $M$ random variables satisfying
  $E[\exp(sY_j)] \leq \exp((s^2\sigma^2)/2)$ for any integer $j$ and
  any $s>0$. Then, we have
  \begin{equation*}
    E[ \max_{1 \leq j \leq M} Y_j] \leq \sigma \sqrt{\log M}.
  \end{equation*}
\end{theorem}


\par



\end{document}